\documentclass[preprint,draft,12pt]{elsarticle}

\usepackage{amssymb,amsmath,amsthm}
\usepackage{color}
\usepackage{algorithm}
\usepackage{algorithmic}

\usepackage{tensor}

\newtheorem{assumption}{Assumption}
\newtheorem{remark}{Remark}
\newtheorem{lemma}{Lemma}
\newtheorem{proposition}{Proposition}
\newtheorem{corollary}{Corollary}
\newtheorem{theorem}{Theorem}

\newcommand{\R}{\mathbb{R}}
\newcommand{\N}{\mathbb{N}}
\newcommand{\M}{\mathbb{M}}
\newcommand{\LL}{\mathcal{L}}

\newcommand{\Nn}{\mathcal{N}}
\newcommand{\dist}{\mathrm{dist}}
\newcommand{\dif}{\mathrm{d}}

\renewcommand{\a}{\textsc{a}}

\newcommand{\Exp}{\mathrm{Exp}}
\newcommand{\id}{\mathrm{id}}
\newcommand{\I}{\mathcal{I}}
\newcommand{\diag}{\operatorname{diag}}

\newcommand{\mgm}{\operatorname{MGM}}

\newcommand{\mgmtrunc}{\operatorname{MGMTRUNC}}
\newcommand{\tgm}{\operatorname{TGM}}

\newcommand{\smoothingop}[1]{\mathcal{S}_{#1}}
\newcommand{\smoothingmat}[1]{\vec{S}_{#1 }}
\newcommand{\smoothingmatpert}[1]{\vec{\check{S}}_{#1}}
\newcommand{\smoothingoppert}[1]{\mathcal{\check{S}}_{#1}}
\newcommand{\smoothingmattrunc}[2]{\check{\vec{S}}_{#1;#2}}
\renewcommand{\vec}[1]{\boldsymbol{\mathbf{#1}}}
\newcommand{\twogridmat}[1]{\vec{C}_{\text{TG}_{#1}}}

\newcommand{\twogridop}[1]{\mathcal{C}_{\text{TG}_{#1}}}
\newcommand{\twogridoppert}[1]{\mathcal{\check{C}}_{\text{TG}_{#1}}}
\newcommand{\twogridmatpert}[1]{\vec{\check{C}}_{\text{TG}_{#1}}}
\newcommand{\multigridmat}[1]{\vec{C}_{#1}}
\newcommand{\multigridmatpert}[1]{\vec{\check{C}}_{#1}}
\newcommand{\multigridop}[2]{\mathcal{C}_{#1}}
\newcommand{\multigridoppert}[1]{\mathcal{\check{C}}_{#1}}

\newcommand{\stiffness}[1]{\vec{A}_{#1}}
\newcommand{\stiffnesspert}[1]{\vec{\check{A}}_{#1}}
\newcommand{\stiffnessprecon}[1]{\vec{B}_{#1}}

\newcommand{\preconpert}[1]{\vec{\check{B}}_{#1}}
\newcommand{\stiffnesstrunc}[2]{\check{\vec{A}}_{#1;#2}}

\newcommand{\prolmat}[1]{\vec{p}_{#1}}
\newcommand{\prolmattrunc}[2]{\check{\vec{p}}_{#1;#2}}
\newcommand{\prolop}[1]{\mathcal{I}_{#1-1}^{(#1)}}
\newcommand{\prolpert}[1]{\check{\vec{p}}_{#1}}
\newcommand{\respert}[1]{\check{\vec{r}}_{#1}}

\newcommand{\resmat}[1]{\vec{r}_{#1}}
\newcommand{\resmattrunc}[2]{\check{\vec{r}}_{#1;#2}}
\newcommand{\resop}[1]{\mathcal{I}_{#1}^{(#1-1)}}

\newcommand{\change}[1]{\textcolor{black}{ #1}}
\newcommand{\crchange}[1]{\textcolor{black}{ #1}}

\newcommand{\CZ}{C_{\mathrm{zeros}}}
\newcommand{\Cen}{C_{\mathrm{EN}}}
\newcommand{\Cpw}{C_{\mathrm{PW}}}
\newcommand{\Cstiff}{C_{\mathrm{stiff}}}
\newcommand{\Cdiag}{C_{\mathrm{diag}}}
\newcommand{\CB}{C_{\mathrm{Bernstein}}}
\newcommand{\CH}{C_{\mathrm{H\ddot older}}}

\newcommand{\Hawaii}{Hawai\kern.05em`\kern.05em\relax i }

\makeatletter
\def\ps@pprintTitle{%
  \let\@oddhead\@empty
  \let\@evenhead\@empty
  \let\@oddfoot\@empty
  \let\@evenfoot\@oddfoot
}
\makeatother

\begin{document}

\begin{frontmatter}




\title{Kernel Multigrid on Manifolds}


\author[label1]{Thomas Hangelbroek
\fnref{fnThomas}}
\ead{hangelbr@math.hawaii.edu}
\address[label1]{University of Hawai`i at M\=anoa}

\author[label2]{Christian Rieger\corref{cor1}%
}
\ead{riegerc@mathematik.uni-marburg.de}
\address[label2]{Philipps-Universit\"at Marburg}

\cortext[cor1]{Corresponding author}
\fntext[fnThomas]{Thomas Hangelbroek's research was supported 
by Grant DMS-2010051 from the National Science Foundation.}


\begin{abstract}
Kernel methods for solving partial differential equations 
on surfaces have the advantage that those methods work intrinsically 
on the surface and yield high approximation rates if the solution 
to the partial differential equation is smooth enough.
Localized Lagrange bases have proven to alleviate
the computational complexity of usual kernel methods to some extent.
The efficient numerical solution of the resulting linear systems of equations 
has not been addressed in the literature so far. 
In this article we apply the framework of geometric 
multigrid method with a $\tau\ge 2$-cycle 
to this particular setting. 
Moreover, we show that the resulting linear algebra can be made more efficiently 
by using the Lagrange function decay again. 
The convergence rates are obtained by a rigorous analysis.
The presented version of a multigrid method provably works 
on quasi-uniform point clouds on the surface 
and hence does not require a grid-structure. 
Moreover, we can show that the computational cost to solve the linear system scales log-linear 
in the degrees of freedom.
\end{abstract}


\end{frontmatter}



\section{Introduction}
The numerical solution to partial differential equations,
often time dependent, 
on curved geometries is crucial to many real-world applications. 
Of the many available numerical methods (see for instance \cite{dziuk_elliott_2013} for an overview using finite elements, 
or \cite{Shankar:etal:2020}  \cite{LeGia:etal:2012} for alternative mesh-free methods), we focus on mesh-free kernel-based Galerkin methods (see also \cite{legia:2004}, \cite{wendland:kuehnemund:2020}, \cite{sun:ling:2022}), which have a number of  merits, including delivering high approximation orders for smooth data, providing smooth solutions, and 
working coordinate-free and without the need for rigid  underlying geometric structures like meshes and regular grids.

Let $\M$ denote in the following a compact, {$d$-dimensional}
 Riemannian manifold without boundary. 
We will consider,
as a  spatial operator,  
a generic second order linear elliptic differential operator $\LL$ with trivial nullspace. 
Such operators occur for instance in the numerical solution of the heat equation using implicit time-stepping methods, see e.g. \cite{KNWW}.
In conventional kernel-based Galerkin methods, a grid $\Xi\subset \M$ is considered and a stiffness matrix $\stiffness{\Xi}$,
which  represents  $\LL$ on a finite dimensional kernel space expressed endowed with  some fixed basis,
 is assembled. 
This leads to the equation 
\begin{equation}\label{eq:stiffness_equation}
\stiffness{\Xi} \vec{u}=\vec{b},
\end{equation}
 (see (\ref{eq:stiffnessmatdef}) for a precise definition).

In this paper we address the problem of reducing the computational costs of using kernels without spoiling their analytic advantages.
Here, we mostly follow \cite{KNWW, NRW} and use the kernel-based Lagrange basis. 
This particular basis, though not locally supported, has very strong decay properties which allows to localize computations. 
The almost-local support already alleviates the problem of densely populated matrices as usually encountered in kernel methods.
Concretely, the full system in (\ref{eq:stiffness_equation}) can be well-approximated by a sparse 
matrix using the decay of
 the local Lagrange basis.
 Despite this,  the condition number $\operatorname{cond}_2 (\stiffness{\Xi}) \sim N^{2/d} $ 
 grows with the problem size $N:=|\Xi|$. 
Thus, even if compressed, (\ref{eq:stiffness_equation}) poses a computational challenge, and, because of its conditioning,
iterative methods (without preconditioning) cannot be expected to work well, since  the iterations needed to ensure a prescribed accuracy 
may grow with the number of degrees of freedom.

To this end, we introduce and analyze a mesh-free multigrid algorithm. 
Specifically, we adapt the standard $W$-cycle multigrid method 
to (localized) kernels and provide a rigorous analysis in this setting.
Although multigrid methods recently have gained attention  in the 
mesh-free community, see \cite{multicloud} and \cite{MGM}, 
the rigorous analysis we provide has been missing in the kernel-based context.

The main novelty of this paper is the following:
\begin{enumerate}
\item We prove this mesh-free algorithm is a contraction, with  norm
independent of grid size (see  Theorem \ref{thm:main}). 
By standard techniques, it follows that numerical solutions within given tolerance
can be obtained with a fixed,
with respect to grid size, number of iterations
(see (\ref{eq:mgm_iteration_count}) and adjacent discussion).
\item  The method we present is stable under perturbations of the stiffness, restriction and prolongation matrices. 
Thus small errors, which may come  from sparsifying these matrices (as well as from quadrature, round-off or modification of the kernel),
also 
yield contractions (see  Theorem \ref{thm:theorem_pert}) and
therefore do not hinder performance of the algorithm.
\item By compressing the  stiffness, 
restriction and prolongation matrices, 
which have   rapid off-diagonal decay, we obtain a  solution method enjoying nearly linear complexity.
\end{enumerate}
This  leads to an approximate solution to the system (\ref{eq:stiffness_equation}) 
which requires only 
$\mathcal{O} \bigl( N \log(N)^{d}\bigr)$
 floating point operations per iteration and where again the multigrid matrix has a norm bound less than unity (see 
 Theorem \ref{thm:theorem_pert}), for
a total operation count of 
$$\mathcal{O} \left( N \log(N)^{d}  \log\left({\epsilon_{\max} }\right)\right),$$
where $\epsilon_{\max}$ is the user-prescribed tolerance.
 The approximate solution
$\vec{u}^{\star}$ to (\ref{eq:stiffness_equation})  has
an (additive) error bound of the form 
$
\|\vec{u}-\vec{u}^{\star}\|_{\ell_2}\le  
\epsilon_{\max} + \epsilon_{\mathrm{tr}}$,
where the error due to truncating matrices
is	
$\epsilon_{\mathrm{tr}}
=\mathcal{O}(N^{-J})
$. 
Here $J>0$ is a user-determined constant which depends linearly on the sparsity of truncated matrices in the algorithm.
This is explained in Remark \ref{remark:final}.

\begin{remark}
Before proceeding, we make the following comments:
\begin{itemize}
\item 
We 
do not attempt to modify the underlying framework of the multigrid method,
as described, for instance, in \cite{hackbusch1994iterative,brenner-scott,reuskenLN},
but show that it can be successfully
applied to kernel methods. This is far from being obvious.
We follow conceptually the book \cite{brenner-scott} 
where also the function space view on multigrid 
methods is used.
\item The fact that $\LL$ is injective is
a convenient simplification we assume throughout the article. An investigation of operators with a non-trivial 
nullspace will be considered in a forthcoming work.
\item Throughout this article, 
we will assume to have access to a sequence of quasi-uniform and nested point clouds
$\Xi_{0} \subset \Xi_1 \subset \dots \Xi_{L} \subset \M$ on $\M$, 
and a corresponding highly localized Lagrange basis generated by a kernel $\phi:\M\times \M\to \R$.
Both, the construction of point clouds on manifolds and the computation of the Lagrange basis are independent of the partial differential equation 
and can hence be pre-computed and even be stored.
\item Although numerical integration is necessary to implement the kernel-based Galerkin method, 
our results hold independently
of the choice of quadrature method. 
(This is treated, for instance, in \cite[Section 4]{KNWW}.) 
We therefore assume for the remainder of this paper these steps to be solved. 
In particular, we assume to have access to the stiffness matrices
%
$\stiffness{\Xi}$, as considered in \cite{FHNWW} or \cite{HNRW-manifold}.
\item 
An alternative approach to treating (\ref{eq:stiffness_equation})
would be to apply a suitable preconditioner. To be effective in this context, such methods would also have to be adapted to the kernel situation 
and the analysis including the compression argument would have to be carefully carried over
as well. 
Furthermore, many successful preconditioners for finite elements, like \cite{BPX}, 
typically involve  concepts from multigrid methods such as hierarchy of approximation spaces.
\item 
 We discuss in this article the multigrid method only from the perspective as standalone solver. We point out that the multigrid method itself provides an attractive preconditioner. Most often, the multigrid preconditioner is combined with a flexible GMRES (FGMRES) iteration method, see e.g. \cite{johnLN} and references therein.
This is also implemented in many software libraries such as PETSc
(see \cite[page 92f.]{petsc-user-ref}) just to name one prominent example. We will discuss this in more detail when we present numerical results on this method in an upcoming work.
\end{itemize}
\end{remark}

\paragraph{Outline of the paper}
The remainder of this article is organized as follows.
In section \ref{sec:setup}, we introduce the basic notation of second order elliptic equations on manifolds, and their 
solution via  kernel-based Galerkin approximation. In this section we  demonstrate the {\em approximation property}
in the  kernel context, which, along with the {\em smoothing property}, provides the analytic backbone for the success of the multigrid method.
Section \ref{sec:Lagrange} introduces the phenomenon of rapidly decaying Lagrange-type bases, which
holds for certain kernels -- using such bases permits stiffness matrices with rapid off-diagonal decay, among other things. 
Of special importance is the diagonal behavior of the stiffness matrix given in Lemma \ref{lem:stiffness_diagonal}, 
which is a novel contribution of this paper, 
and which is the  analytic result necessary to prove the {\em smoothing property}.
In  section \ref{sec:smoothing_prop} we discuss the smoothing property of damped 
Jacobi iterations for kernel-based methods both in the case of symmetric 
and non-symmetric differential operators. 
In section \ref{sec:direct_mgm} we introduce kernel-based restriction and prolongation operators,  the standard two-grid method
and then the $W$-cycle. In this section, we prove Theorem \ref{thm:main}, a consequence of which is the bound (\ref{eq:mgm_iteration_count}), which shows the (poly)-logarithmic complexity of our proposed method.
Section \ref{sec:perturb} treats the  error resulting from small perturbations of the stiffness matrix, as well as the prolongation and restriction matrices.
The main result in this section, Theorem \ref{thm:theorem_pert}, demonstrates how such errors affect the multigrid approximation error.
Section \ref{sec:trunc} investigates the computationally efficient  truncated multigrid method as an application of the previous section.

\section{Problem Set Up} \label{sec:setup}
\subsection{Manifold}
Consider a compact  $d$-dimensional Riemannian manifold without boundary $\mathbb{M}$.  Here we list
some useful tools and their properties which hold in this setting. We direct the reader to \cite{AMR, Lee} for 
relevant background. Results about covariant derivatives and Sobolev spaces can be found in \cite{HNW}.
%
%

The tangent bundle is $T\M$  and the cotangent bundle is $T^*\M$.
We denote  by $T^{k}_{r}\M$ the vector
bundle of tensors  with contravariant order $k$ and covariant order $r$; 
we say {\em type} $(k,r)$ for short.
Thus
$T\M=T^{1}_{0} \M$  and $T^*\M=T^{0}_{1}\M $.
We will denote the fiber at $x\in \M$ by $T^{k}_{r}\M_x$.
The space of  tensor fields of type $(k,r)$ 
(known also as sections; i.e.,  maps $\mathbf{S}:\M\to T_r^k\M$
with $\mathbf{S}(x)\in T_r^k\M_x$ 
for every $x\in \M$)
 is denoted $\mathcal{T}_r^k\M$.

In this article, we are concerned primarily with covariant tensors (i.e., tensors of type $(0,k)$), so we 
use the short hand notation $T_k\M = T_k^0\M$
and $\mathcal{T}_k\M=\mathcal{T}_k^0\M$. 

For a chart $(U,\phi)$ for $\M$ from which
we get the usual vector fields $\frac{\partial}{\partial x^j}$
and forms $d x^j$ ($1\le j\le d$), which act as local bases for $T\M$ and $T^*\M$ over $U$.
These can be used to generate bases for tensor fields. In particular, for
 given covariant rank $k$
 and  $\vec{i}=(i_1,\dots, i_k)\in \{1,\dots,d\}^k$
 we have basis 
element $dx^{\vec{i}}:= dx^{i_1} \cdots  dx^{i_k}$ .
This 
allows us to write $\mathbf{S}\in \mathcal{T}_k\M$ in coordinates as
$\mathbf{S} (x)= \sum_{\vec{i}\in \{1,\dots,d\}^k} ( S(x))_{\vec{i}} dx^{\vec{i}}$. 

Because $\M$ is a Riemannian manifold,  at each
$x\in\M$, $T\M_x$ has an inner product $\langle \cdot,\cdot \rangle_x$ and induced norm $\|\cdot\|_x$. This means that
there is by $g\in \mathcal{T}_2\M$ (the metric tensor), so that for  tangent vectors in $V,W \in T\M_x$ 
we have
$g(x)(V,W) =\langle V,W\rangle_{x} $.
The inner product extends to the dual: for cotangent vectors $\mu,\nu \in T^*\M_x$ 
we have $\langle \mu,\nu\rangle_{x} = \sum \mu_j \nu_k g^{j,k}$, where
$\sum g_{j,k} g^{k,\ell} = \delta_{j,\ell}$. From this, it naturally extends to tensors.
 For $\mathbf{T},\mathbf{S} \in {T}_k\M_x$, 
 written in coordinates as
 $\mathbf{T}= \sum_{\vec{j}\in\{1,\dots,d\}^k}T_{\vec{j}} \, d{x^{j_1}}\dots d{x^{j_k}}$
 and $\mathbf{S}= \sum_{\vec{i}\in\{1,\dots,d\}^k}S_{\vec{i}} \, d{x^{i_1}}\dots d{x^{i_k}}$,
 we have
 $$
 \langle \mathbf{T},\mathbf{S}\rangle_{x} 
 = 
 \sum_{\vec{i},\vec{j}\in\{1,\dots,d\}^k} 
    g^{i_1 j_1}(x)\dots g^{i_k,j_k}(x) S_{\vec{i}} T_{\vec{j}}.
 $$

We denote the Riemannian distance on $\M$ by 
$\dist:\M\times \M \to [0,\infty)$; it is given by the formula
$\dist(x,y) = \inf\{ \int_a^b \|\gamma'(t)\|_{\gamma(t)} \dif t \mid \gamma(a)=x,\ \gamma(b)=y\}$
where the infimum is taken over piecewise smooth curves connecting $x$ and $y$.
The Riemannian metric gives rise to a volume form $\dif \mu = \sqrt{\det(g(x))} \dif x$.
By compactness, there exist
constants $0<\alpha_{\M} \le \beta_{\M}$ so that
$$\alpha_{\M} r^d\le  \mu(B(x,r))\le \beta_{\M} r^d$$
or any ball $B(x,r):= \{y\in \mathbb{M}\mid \dist(x,y)<r \}$ centered at $x\in\M$ and having radius $0<r<\mathrm{diam}(\M)$.

For a finite subset $\Xi\subset \M$, we can define the following useful quantities:
the {\em separation distance}, $q$ of $\Xi$  
and the
{\em fill distance}, $h$, of $\Xi$ in $\M$. They are given by
$$q:=\frac12\min_{\zeta\in\Xi} \dist(\zeta,\Xi\setminus\{\zeta\})
\quad\text{and} \quad
h := h(\Xi,\M):= \sup_{x\in\M} \dist(x,\Xi).$$
The finite sets considered throughout this paper will be quasiuniform, with mesh ratio
$\rho:=h/q$ bounded by a fixed constant. For this reason, quantities which are controlled above or below
by a power of $q$ can be likewise controlled by a power of $h$ -- in short, whenever
possible, we express estimates in terms of the fill distance
$h$,  allowing constants to depend on $\rho$.
For instance, the cardinality $|\Xi|$   is bounded above and below by 
\begin{equation}\label{eq:numberofpoints}
\frac{\mu(\M)}{\beta_{\M}}h^{-d} \le |\Xi| \le \frac{\mu(\M)}{\alpha_{\M}}q^{-d}. 
\end{equation}
Similarly, if $f:[0,\infty) \to \R$ is continuous then for any $x\in\M$,
\begin{equation}\label{eq:finite_sum}
\sum_{\zeta\in\Xi} f\bigl(\mathrm{dist}(\zeta,x)\bigr)
\le \max_{t\le q} f(t) +\frac{\beta_{\M}}{\alpha_{\M}} \sum_{n=1}^{\infty} (n+2)^d \max_{nq\le t\le (n+1)q}f(t).
\end{equation}

Moreover, we use the following notation: $\R^{A}=\{f:A \to \R \}$ denotes the functions from the set $A$ to $\R$. Of course, if $A$ is a discrete finite set, we can identity $\R^A \cong \R^{|A|}$.

\subsection{Sobolev spaces}\label{subsec:Sobolev}
%
 The covariant derivative $\nabla$
maps tensor fields  of  type $(r,s)$  to  fields of type $(r,s+1)$. Its adjoint (with respect to the $L_2$ inner products
on the space of sections of $T_r^s(\M)$) is denoted $\nabla^*$.
For functions, this is fairly elementary. The covariant derivative of a  (scalar) function $f:\M\to \R$ 
equals its exterior derivative; in coordinates, we have
$$
\nabla f = 
\sum_{j=1}^d \frac{\partial f}{\partial x^j} dx^j 
=df .
$$
For a 1-form $\omega = \sum_{j=1}^d \omega_j dx^j$, we have
$$
\nabla^* \omega(x)
=  
-\frac{1}{\sqrt{ \det g(x)} } 
  \sum_{j=1}^d \sum_{k=1}^d \frac{\partial }{\partial x_k} 
    \left( \sqrt{ \det g(x)} g^{jk}(x) \omega_j(x)\right).
$$
A direct calculation  shows  $\int_{\M}f(x) \nabla^* \omega (x)\dif \mu(x) = \int_{\M} \langle\omega(x),\nabla f(x)\rangle_x\dif \mu(x)$.

 For $\Omega\subset \M$ Sobolev space $W_p^k(\Omega)$ is defined to be the set of functions $f:\Omega\to \R$ which satisfy 
 $$\| f\|_{W_p^k(\Omega)}^p = \sum_{\ell=0}^k \int_{\Omega} \| \nabla^\ell f\|_{x}^p \dif \mu(x)<\infty,$$
where the $\ell$-th order covariant derivatives can be found in \cite[Section 2.2]{HNW}.

For a scalar function $f:\M\to \R$, 
the Laplace-Beltrami operator is given in local coordinates as
$\Delta f = 
\sum_{j=1}^d
\sum_{k=1}^d 
\frac{1}{\sqrt{|\det g|}}
\frac{\partial}{\partial x^j}
 \bigl( \sqrt{|\det g|} g^{jk} \frac{\partial}{\partial x^k } f\bigr)
 $.
 Thus for scalar functions $\Delta f = -\nabla^* \nabla f$.
For any integer $k\in \N$ and $p\in(1,\infty)$,
the Bessel-potential norm $\|(1-\Delta)^{k/2} f\|_{L_p(\M)}$ is equivalent to $\|f\|_{W_p^k(\M)}$, as demonstrated in
\cite[Theorem 4(ii)]{Trieb}
(although when $k=1$ and $p=2$, the two norms are equal;  this can be observed directly).

Lemma 3.2 from \cite{HNW} applies, so there are uniform constants $\Gamma_1,\Gamma_2$  and 
$r_{\M}>0$
so that the family of exponential maps
$\{\Exp_x: B(0,r_{\M})\to \M\mid x\in\M\}$ (which are diffeomorphisms taking $0$ to $x$)
provides local metric equivalences: for any open set $\Omega\subset B(0,r_{\M})\subset \R^d$, we have 
\begin{equation}\label{eq:expmap}
\Gamma_{1} \|u\circ\Exp_x\|_{W_p^j(\Omega)}\le \|u\|_{W_p^j(\Exp_x(\Omega))}\le \Gamma_{2} \|u\circ\Exp_x\|_{W_p^j(\Omega)}.
\end{equation}

This shows that $W_p^k(\M)$ can be endowed with equivalent norms using a partition of unity $(\varphi_j)_{j\le N}$
 subordinate  to a cover $\{\mathcal{O}_j\}_{j\le N}$ with associated  charts $\psi_j: \mathcal{O}_j\to \R^d$ to obtain
$\|u\|_{W_p^k(\M)}^p \sim \sum_{j=1}^N \|(\varphi_j u)\circ \psi_j^{-1}\|_{W_p^k(\R^d)}^p$. Here constants of equivalence
depend on the partition of unity and charts.

 A useful result in this setting, which we will use explicitly in this article, but is also behind a number 
 of background results in section \ref{sec:Kernel_Setup}, 
 is the following zeros estimate \cite[Corollary A.13 ]{HNW-p}, 
 which holds for Sobolev spaces:
 If $u \in W_2^m(\M)$ satisfies $u\left|_{X}\right. = 0$, then for any $k\le m$  we have
 \begin{equation}\label{eq:eqn_zeros_estimate}
 \|u\|_{W_2^k(\M)} \le \CZ h^{m-k}\|u\|_{W_2^m(\M)}.
 \end{equation}
 Here $\CZ$
 depends on $m$ and  $\M$.
The result can also be obtained on bounded, Lipschitz domain $\Omega \subset \M$
that satisfies a uniform cone condition, with the cone having radius $R_{\Omega} \le  r_{\M}/3$,
although the constant $\CZ$
will depend on the aperture of the cone in that case.
See \cite[Theorem A.11]{HNW-p} 
for a precise statement and definitions of the involved quantities.

\subsection{Elliptic operator}
\label{subsec:EllipticOperator}
 
We write the operator $\LL$ in divergence form:
$$
\LL= \nabla^* \a_2^{\flat}(\nabla\cdot ) + \a_1 \nabla  +\a_0.
$$
Here $\a_0$ is a smooth function, $\a_1$ is a smooth tensor field of type $(1,0)$ 
and $\a_2$ is a smooth tensor field of  type  $(2,0)$ 
which generates the field
$\a_2^{\flat}$ of type  (1,1). 
In coordinates, $\a_2$ has the form $\sum_{j=1}^d \sum_{k=1}^d a^{jk}\frac{\partial}{\partial x_j} \frac{\partial}{\partial x_k}$,
and 
$\a_2^{\flat}$ is $\sum_{j=1}^d \sum_{k=1}^d  \tensor{a}{_k^j}  dx_k \frac{\partial}{\partial x_j}$
with $\tensor{a}{_k^j} = \sum_{\ell=1}^d g_{k\ell} a^{\ell j}$.
Here we have used the index lowering operator $^\flat$ which ensures, for any $\mu,\nu\in T_1\M_x$, that
$\langle \a_2^{\flat}( \mu,\cdot), \nu\rangle_x = \a_2(\mu,\nu)$. 
 This follows
by a direct calculation from the above expressions in  coordinates (index lowering and the $^\flat$ operator are discussed in
 \cite[p.\ 341]{AMR} or \cite[p.\ 27]{Lee}).

Furthermore, we require $c_0>0$ to be a constant
so that
\begin{equation}
\label{eq:ellipticity}
\a_2(x)(v, v)\ge c_0\langle v,v\rangle_x
\quad 
\text{and} 
\quad \a_0+\frac12 \nabla^*\a_1 \ge c_0.
\end{equation}
As a basic example, we consider $\a_2(x)=\langle\cdot,\cdot\rangle_x$, $\a_1=0$ and $\a_0=1$; then (\ref{eq:ellipticity})  holds with $c_0=1$.
In this case, 
$a^{jk} = g^{jk}$,
$\tensor{a}{_k^j}=\delta_{j,k}$ and $\LL=1+\nabla^*\nabla = 1-\Delta$.

With the identity
$\int_{\M}u (\a_1 \nabla u)  =
\frac12 \int_{\M}(\a_1 \nabla( u^2))  
 = \frac12\int_{\M} (\nabla^*\a_1)u^2$,
the second part of (\ref{eq:ellipticity}) 
  ensures that
$$
\int_{\M} u(\a_1 \nabla u + \a_0 u)\ge c_0 \|u\|_{L_2(\M)}^2.$$
We have also that
$\int_{\M}u(x) \nabla^*(\a_2^\flat \nabla u)(x) \dif \mu(x) = \int_{\M} \langle\a_2^\flat \nabla u , \nabla u\rangle_x \dif \mu(x) $.
The definition of $\a_2^\flat $ ensures that this equals
$  \int_{\M} \a_2(x)\bigl( \nabla u (x), \nabla u(x)\bigr) \dif \mu(x) $, and the first 
part of (\ref{eq:ellipticity}) then ensures that
$$
 \int_{\M} \nabla^*(\a_2^\flat \nabla u)(x) \dif \mu(x) \ge
  \int_{\M} \|\nabla u\|_x^2\dif\mu(x).$$

Thus, (\ref{eq:ellipticity}) 
 guarantees that the bilinear form $a( u,v):= \int_{\M} v\LL u $, 
defined initially for smooth functions,
 is bounded on $W_2^1(\M)$ and 
is coercive. 
\change{Thus, we have that the energy quasi-norm $[u]^2_{\LL}:=a(u,u)$ satisfies the metric equivalence
\begin{align}\label{eq:energynorm}
	c_0 \|u\|_{W_2^1(\M)}^2 \le [u]^2_{\LL}=a(u,u)\le C \|u\|_{W_2^1(\M)}^2.
\end{align}
(if $a$ is symmetric, this is a norm, and we write  $\|u\|_{\LL}=[u]_{\LL}$.)}

Please note 
that this excludes differential operators with a null space at the moment. Having in mind the time dependent problems, this assumption is justified. The technical more challenging analysis for those more general operators is left to future research.

\subsection{Galerkin methods}
We fix $f \in L_{2}(\M)$ and consider $u \in W^{1}_{2}(\M)$ as solution to 
\begin{align}\label{eq:weakform}
	a( u,v):=
	\int_{\M} v \LL u
	=(f,v)_{L_{2}(\M)}=F(v)  \quad \text{for all } v \in W^{1}_{2}(\M).
\end{align}
Regularity estimates (\cite[Chapter 5.11, Theorem 11.1]{Taylor}) yield that 
$u \in W^{2}_{2}(\M)$
and $\| u\|_{W^{2}_{2}(\M) } \le C \|f\|_{L_{2}(\M)}$.
Consider now a family of finite dimensional subspaces  $(V_h)$ with
$V_h\subset W_2^1(\M)$ 
associated to a parameter\footnote{In the
sequel, these will be kernel spaces generated by a subset $\Xi\subset \M$, and $h=h(\Xi,\M)$ will be the fill distance.
For now, we consider a more abstract setting, with $h>0$ only playing a role in establishing the approximation
property $\dist_{W_2^1}(g,V_h) \le C h \|g\|_{W_2^2(\M)}$ below.} $h>0$.
Define $P_{V_h}: W^{1}_{2}(\M)  \to V_h$ so that
\begin{align*}
a( P_{V_h }u, v)=
	a( u,v) \quad \text{for all } v \in V_h.
\end{align*}
\change{Because $a\left(  u - P_{V_h}u, P_{V_h}u-v\right)=0$, the classical C{\'e}a Lemma holds:}
\begin{align*}
	c_1 \left\| u - P_{V_h}u \right\|^2_{W^{1}_{2}(\M)} 
	&\le 
	a\left( u  - P_{V_h}u, u - P_{V_h}u \right)
	=\change{ a\left(  u - P_{V_h}u, u-v\right)} \\
	& \le c_2 
	\left\| u - P_{V_h}u  \right\|_{W^{1}_{2}(\M)}\left\| u -v  \right\|_{W^{1}_{2}(\M)} \quad \text{ for all } v\in V_h.
\end{align*}
And thus we get
\begin{align}\label{eq:Cea}
	 \left\| u - P_{V_h} u \right\|_{W^{1}_{2}(\M)} \le \frac{c_2}{c_1} \dist_{
	 {W^{1}_{2}(\M)}}\left(u,V_h \right).
\end{align}
This can be improved by a {\em Nitsche}-type argument to get the following result, whose proof can be found in many textbooks on numerical methods for partial differential equations. 
\begin{lemma}\label{lem:Nitsche}
	Suppose the family $(V_h)$ has the property that
	for all $\tilde{u}\in W_2^2(\M)$, 
	the  distance 
	in $W_2^1(\M)$
	 from $V_h$ satisfies
	$\dist_{\|\cdot\|_{W^{1}_{2}(\M)}} \left(\tilde{u},V_h \right) 
	 \le 
	C  h \left\| \tilde{u} \right\|_{W^{2}_{2}(\M)} 
	$.
	Then for any
	 $u \in W_{2}^{2}(\M)$,
	\begin{align}
\label{eq:Nitsche}
	\left\| u-P_{V_h }u  \right\|_{L_{2}(\M)} 
	\le  
	C  h  \dist_{\|\cdot\|_{W^{1}_{2}(\M)}}\left(u,V_h \right) .
\end{align}
\end{lemma}

We point out, that Lemma \ref{lem:Nitsche} does not depend on the choice of a specific basis in the finite dimensional space $ V_h $.

\subsection{Kernel approximation}\label{sec:Kernel_Setup}
We  consider a continuous
function $\phi:\M \times \M\to \R$,  the  {\em kernel},
which  satisfies a number of analytic properties which we explain in this section.

Most important is that $\phi$ is {\em conditionally positive definite} with respect to some 
(possibly trivial) finite
dimensional subspace\footnote{generally a space spanned by some
eigenfunctions of the Laplace-Beltrami operator} $\Pi\subset C^{\infty}(\M)$.
Conditional positive definiteness with respect to $\Pi$ means that for any  $\Xi\subset \M$,
the collocation matrix $\bigl(\phi(\xi,\zeta)\bigr)_{\xi,\zeta}$ is positive definite on the vector
space $\{a\in \R^{\Xi}\mid  (\forall p\in \Pi)\, \sum_{\xi\in\Xi} a_{\xi} p(\xi) = 0\}$. 
As a result, 
if $\Xi$  separates elements of $\Pi$, then the space 
\begin{align}\label{eq:trialspace}
V_{\Xi}:=\left\{\sum_{\xi \in \Xi} a(\xi)\phi(\cdot,\xi)\mid  (\forall p\in \Pi)\, \sum_{\xi\in\Xi} a_{\xi} p(\xi) = 0\right\}+\Pi
\end{align}
has dimension $\Xi$.
In case $\Pi=\{0\}$, the kernel is {\em positive definite}, and the collocation matrix is strictly positive definite on $\R^{\Xi}$,
and $V_{\Xi} = \mathrm{span}_{\xi\in \Xi} \phi(\cdot,\xi)$.

For a conditionally positive definite kernel, there is an associated reproducing kernel semi-Hilbert space
$\Nn(\phi)\subset C(\M)$ with the property that $\Pi = \mathrm{null}(\|\cdot\|_{\Nn(\phi)})$ and 
that for any $a\in \R^{\Xi}$ for which $ \sum_{\xi\in\Xi} a_{\xi} \delta_{\xi}\perp \Pi$, the following identity
$$
\sum_{\xi \in \Xi} a_{\xi} f(\xi) = \langle f, \sum_{\xi \in \Xi} a_{\xi} \phi(\cdot,\xi)\rangle_{\Nn(\phi)}
$$ holds for all 
$f\in\Nn(\phi)$.
It follows that if $\Xi$  separates elements of $\Pi$, interpolation  with $V_{\Xi}$ is well defined on $\Xi$
and the projection $I_{\Xi} :\Nn(\phi)\to V_{\Xi}$ is orthogonal with respect to the semi-norm on $\Nn(\phi)$.
Of special interest are (conditionally) positive definite kernels that
have Sobolev native spaces.
\begin{lemma}
If $\phi$ is conditionally positive definite with respect to $\Pi$ 
and satisfies the equivalence $\Nn(\phi)/\Pi\cong W_2^m(\M)/\Pi$.
then there is a constant $C$ so that for 
 any integers $k,j$ with $0\le k\le j\le m$
and  any $u \in W_2^j(\M)$, 
 $$\dist_{W_2^k(\M)}(u,V_{\Xi})\le Ch^{j-k}\|u \|_{W_2^j(\M)}.$$
\end{lemma}

\begin{proof}
The case $k=j$ is trivial; it follows by considering $0\in V_{\Xi}$.

For $j=m$ and $0\le k<m$, 
the  zeros estimate \cite[Corollary A.13]{HNW-p} 
ensures 
that
$$\|I_{\Xi} u-u\|_{W_2^k(\M)}\le C h^{m-k} \|I_{\Xi}u-u\|_{W_2^{m}(\M)}$$
 (because the interpolation error $I_{\Xi}f-f$ vanishes on $\Xi$).   
 The hypothesis then gives, 
$\|I_{\Xi} u-u\|_{W_2^k(\M)}\le C h^{m-k} \|I_{\Xi}u - u\|_{\Nn(\phi)}$,
with a suitably enlarged constant.
Because $I_{\Xi}$ is an orthogonal projector, $\|I_{\Xi}u - u\|_{\Nn(\phi)}\le \|u\|_{\Nn}(\phi)$, 
so we have (again enlarging the constant)
 that $$ \|I_{\Xi} u-u\|_{W_2^k(\M)}\le Ch^{m-k}\|u\|_{W_2^m(\M)}.$$

For $0\le k< j<m$, we use the fact that
$W_2^j(\M)$ is the real interpolation space $[W_2^k(\M),W_2^m(\M)]_{\frac{j-k}{m-k},2}$.
This is \cite[Theorem 5]{Trieb}. 
For  $\tilde{u}\in W_2^j(\M)$,
this means that  the $K$-functional  
$$K(\tilde{u},t) = \inf_{g\in W_2^m(\M)} \|\tilde{u}-g\|_{W_2^k(\M)} +t\|g\|_{W_2^m(\M)}$$
satisfies the condition 
$\int_0^{\infty} (t^{-\frac{j-k}{m-k}} K(\tilde{u},t))^2 \frac{\dif t}{t} <\infty$.
Since $K(\tilde{u},t)$ is continuous and monotone, we have that 
$t\mapsto t^{-\frac{j-k}{m-k}} K(\tilde{u},t)$ is bounded on $(0,\infty)$.
Thus for $t = h^{m-k}$, there exists  $g\in W_2^m(\M)$ so that
\begin{align*}
 \|\tilde{u} -g\|_{W_2^k(\M)}
 &\le C h^{j-k} \|\tilde{u} \|_{W_2^j(\M)}
&\text{and}&
 &
\|g\|_{W_2^m(\M)} 
&\le Ch^{j-m}\|\tilde{u} \|_{W_2^j(\M)}.
\end{align*}
The above estimate gives
$\|I_{\Xi}g-g\|_{W_2^k(\M)}\le Ch^{m-k}\|g\|_{W_2^m(\M)}$.
This  implies that $\|\tilde{u}-I_{\Xi}g\|_{W_2^k(\M)}\le C h^{j-k} \|\tilde{u} \|_{W_2^j(\M)}$
as desired.
\end{proof}

As a consequence, the kernel Galerkin solution $u_{\Xi} \in V_{\Xi}$
to $\LL u=f$ with $f\in W_2^{j+1}(\M)$ satisfies
$\|u- u_{\Xi}\|_{W_2^1(\M)}\le C h^{j-1}\|f\|_{W_2^{j+1}(\M)}$.
More importantly, for our purposes, the hypothesis of Lemma \ref{lem:Nitsche} is satisfied by
the space $V_{\Xi}$.
Indeed, by (\ref{eq:Nitsche}), we have the following {\em approximation property}:
\begin{equation}\label{eq:approximation_property}
\|u-P_{\Xi} u\|_{L_2(\M)} \le C h \|u\|_{W_2^1(\M)}
\end{equation}
which, together with a {\em smoothing property}, forms the backbone of the convergence theory for 
the multigrid method.
\section{The Lagrange basis and stiffness matrix}\label{sec:Lagrange}
For a kernel $\phi$ and a set $\Xi$ which separates points of $\Pi$, the Lagrange basis
$(\chi_{\xi})_{\xi\in\Xi}$ for $V_{\Xi}$
satisfies $\chi_{\xi}(\zeta) = \delta_{\xi,\zeta}$ for each $\xi,\zeta\in \Xi$.

A natural consequence of $\Nn(\phi) \cong W_2^m(\M)$ is that 
there exists a constant $C$ so that for any $\Xi\subset \M$, 
the bound $\|\chi_{\xi}\|_{W_2^m(\M)} \le C q^{\frac{d}{2}-m}$ holds. 
\crchange{Especially on Riemannian manifolds without boundary, much stronger statement can be proven, see \cite{HNW,HNSW,HNW-p} or \cite{FHNWW} on $\mathbb{S}^2$. We are, however, not convinced that the list of settings where stronger results hold is complete. In order to allow our results to be applied in future settings where those statement will be shown, we will formulate those stronger statements as assumptions or building blocks.}

A stronger result is the following:
\begin{assumption}\label{assumption:lagrange_decay}
We assume that there is $m>d/2+1$ so that $\Nn(\phi)\cong W_2^m(\M)$,
and, furthermore,
there exist constants $\nu>0$ and $\Cen$ so that for 
$R>0$
$$\|\chi_{\xi}\|_{W_2^m(\M\setminus B(\xi,R))} \le \Cen q^{\frac{d}{2} -m}e^{-\nu \frac{R}{h}}.$$
\end{assumption}
This gives rise to a number of analytic properties, some of which we present here
(there are many more, see \cite{HNSW} and \cite{HNRW-manifold} for a detailed discussions).
For the following estimates, the constants of equivalence depend on 
$\Cen$, $\nu$, and $\rho$.

{\em Pointwise decay:} there exist constants $C$ and $\nu>0$ so that for any $\Xi\subset \M$, the estimate
\begin{equation}\label{eq:ptwise_decay}
|\chi_{\xi}(x)|\le \Cpw e^{-\nu \frac{\dist(x,\xi)}{h}}
\end{equation}
holds. 
Here $\Cpw \le \rho^{m-d/2} \Cen$, where we recall that
the {\em mesh ratio} is
$\rho=h/q$.  

{\em H{\"o}lder  continuity:} Of later importance, we mention the following condition, which follows from \cite[Corollary A.15]{HNW-p}.
For any $\epsilon<m-d/2$, the Lagrange function is $\epsilon$ H{\"o}lder continuous, and
satisfies the bound
\begin{equation}\label{eq:holder}
|\chi_{\xi} (x) -\chi_{\xi}(y)| \le \CH \bigl( \dist(x,y)\bigr)^{\epsilon}h^{-\epsilon}. 
\end{equation}

Although we make explicit the dependence on $\rho$ here, for the
remainder of this article, we assume that constants that follow depend on $\rho$.

{\em Riesz property:} there exist constants $0<C_1\le C_2<\infty$ so that for any $\Xi\subset \M$, and $a\in \R^{\Xi}$, we have
\begin{equation}\label{eq:riesz}
C_1 q^{d/2}\|a\|_{\ell_2(\Xi)} \le \|\sum_{\xi\in \Xi} a_{\xi} \chi_{\xi}\|_{L_2(\M)} \le C_2 q^{d/2} \|a\|_{\ell_2(\Xi)}.
\end{equation}

{\em Bernstein inequalities:} There is a constant $\CB$ so that for $0\le k\le m$,
\begin{equation}\label{eq:Bern_disc}
 \|\sum_{\xi\in \Xi} a_{\xi} \chi_{\xi}\|_{W_2^k(\M)}
  \le 
  \CB h^{d/2-k} \|a\|_{\ell_2(\Xi)}.
\end{equation}

\subsection{The stiffness matrix}
We now discuss the stiffness matrix and some of its properties. Most of these have appeared in
\cite{collins2021kernel}, with earlier versions for the sphere appearing in \cite{HNW} and \cite{NRW}.

A consequence of the results of this section is
that the problem of calculating the Galerkin solution to $\LL u=f$ from $V_{\Xi}$
involves treating a problem whose condition number grows like $\mathcal{O}(h^{-2})$ -- this 
is the fundamental issue that the multigrid method seeks to overcome.

The {\em analysis map} for 
$\bigl(\chi_{\xi}\bigr)_{\xi\in\Xi_{\ell}}$ with respect to the 
bilinear form $a$ defined in \eqref{eq:weakform}
is
\begin{align*}
	\sigma_{\Xi}: 
	\left( W^{1}_{2}(\M) , a(\cdot,\cdot)\right) 
	\to
	 \left(\R^{\Xi} ,(\cdot,\cdot)_{2}\right)
	 : \quad 
	 v \mapsto 
	\bigl(
	a(v,  \chi_{\xi})\bigr)_{\xi\in\Xi}^{T}.
\end{align*}
The analysis map is a surjection.

The {\em synthesis map} is
\begin{align*}
	\sigma^{\ast}_{\Xi}:
	\left(\R^{\Xi} ,(\cdot,\cdot)_{2}\right)
	\to 
	\left( W^{1}_{2}(\M) , a(\cdot,\cdot)\right) 
	: \quad
	  \vec{w} \mapsto \sum_{\xi\in\Xi} w_{\xi} \chi_{\xi}.
\end{align*}
The range of the synthesis map is clearly $V_{\Xi}$; 
in other words, it is the natural isomorphism between Euclidean space and the finite dimensional kernel space; indeed, 
(\ref{eq:riesz}) shows that it is bounded above and below between $L_2(\M)$ and $\ell_2(\Xi)$.
By abusing notation slightly, 
we write $\bigl(\sigma^{\ast}_{\Xi}\bigr)^{-1}: V_{\Xi}\to \R^{\Xi}$.
This permits a direct matrix representation of linear operators on $V_{\Xi}$ via conjugation:
$S \mapsto \mathbf{S}:= (\sigma_{\Xi}^*)^{-1} S \sigma_{\Xi}^*\in \R^{\Xi\times \Xi}$.  Furthermore, by the Riesz property (\ref{eq:riesz}), we have
\begin{align}\label{eq:rieszconsequence}
 \frac{C_1}{C_2}\|\mathbf{S}\|_{\ell_2(\Xi)\to\ell_2(\Xi)}  \le  \|S\|_{L_2(\M)\to L_2(\M)} \le \frac{C_2}{C_1} \|\mathbf{S}\|_{\ell_2(\Xi)\to \ell_2(\Xi)}.
 \end{align}

A simple calculation shows that
$	
a\left(\sigma^{\ast}_{\Xi}(\vec{w}),v\right)
	=
\left\langle \vec{w},\sigma_{\Xi}(v)\right\rangle_2
$, 
so $\sigma_{\Xi}^\ast$ is the $a$-adjoint of $ \sigma_{\Xi}$.
Of course, when $a$ is symmetric, we also have 
$\left\langle  \sigma_{\ell}(v),\vec{w}\right\rangle_2= a\left( v,\sigma^{\ast}_{\ell}(\vec{w})\right)$.

The stiffness matrix is defined as
\begin{align}\label{eq:stiffnessmatdef}
\stiffness{\Xi}:=(A_{\xi,\eta})_{\xi,\eta \in \Xi}:=
\Bigl( a\bigl( \chi_{\xi},\chi_{\zeta}\bigr)\Bigr)_{\xi,\zeta\in\Xi}.
\end{align}
It  represents the operator $\mathcal{L}$
on the finite dimensional space $V_{\Xi} $.
Using the analysis and synthesis maps, 
$\stiffness{\Xi}
	=
	\sigma_{\Xi}\circ  \sigma^{\ast}_{\Xi}
	: \left(\R^{\Xi} ,(\cdot,\cdot)_{2}\right)  
	\to 
	\left(\R^{\Xi} ,(\cdot,\cdot)_{2}\right) : 
	\vec{c}\mapsto \left(a( \chi_{\xi}, \chi_{\zeta})\right)_{\xi,\zeta\in\Xi} \vec{c}
$, 
and we have that the Galerkin projector  
$P_{\Xi}: W_2^1(\M) \to V_{\Xi}$ 
can be expressed as 
$P_{\Xi}=\sigma^{\ast}_{\Xi} \left(\sigma_{\Xi}\circ  \sigma^{\ast}_{\Xi} \right)^{-1}\sigma_{\Xi}$.

\begin{lemma}\label{lem:diag_decay} There is a constant $\Cstiff$ so that the entries of the stiffness matrix satisfy
\begin{equation*}
|A_{\xi,\eta}|
\le 
 \Cstiff h^{d-2} e^{-\frac{\nu}{2} \frac{\dist(\xi,\eta)}{h}}.
\end{equation*}
\end{lemma}
\begin{proof}
For a tensor field $F\in \mathcal{T}_1\M$, we write $|F|:\M\to \R: x\mapsto \|F(x)\|_x$.
Thus using $|\nabla\chi_{\xi}|(x) = \|\nabla\chi_{\xi}(x)\|_x$,
we can bound the integral 
\begin{align*}
|a( \chi_{\xi},\chi_{\eta})|
&\le 
\|\a_2\|_{\infty} \langle |\nabla \chi_{\xi}| ,|\nabla \chi_{\eta}|\rangle_{L_2(\M)}
+ 
\|\a_1\|_{\infty} \langle |\nabla \chi_{\xi}| ,| \chi_{\eta}|\rangle_{L_2(\M)}
\\&+\|\a_0\|_{\infty}\langle |\chi_{\xi}|, |\chi_{\eta}|\rangle_{L_2(\M)}.
\end{align*}
Decompose each inner product  %
using the half spaces $H_+:=\{x\mid\dist(x,\xi)<\dist(x,\eta)\}$ and $H_-=\M\setminus H_+$, noting
$H_+\subset \M\setminus B(\eta,R)$ and $H_-\subset \M\setminus B(\xi,R)$, with $R= \dist(\xi,\eta)/2$.
By applying Cauchy-Schwarz to 
each integral gives, after combining terms,
$$|A_{\xi,\eta}| 
\le C_{\LL}(\|\chi_{\xi} \|_{W_2^1(\M)} \| \chi_{\eta}\|_{W_2^1(\M\setminus B(\eta,R))}
+
\| \chi_{\xi} \|_{W_2^1(\M\setminus B(\xi,R))} \| \chi_{\eta}\|_{W_2^1(\M)})
$$
for some constant $C_{\LL}$ depending on the coefficients of $\LL$.

We  have
$\| \chi_{\xi} \|_{W_2^1(\M)} \le \CB h^{d/2-1}$
by the Bernstein inequality (\ref{eq:Bern_disc}) (with a similar estimate for $\chi_{\eta}$).
The zeros estimate for complements of balls, 
\cite[Corollary A.17]{HNW-p},
applied to $\chi_{\xi}$ gives
 $$\| \chi_{\xi} \|_{W_2^1(\M\setminus B(\xi,R))} \le \CZ h^{m-1}\| \chi_{\xi} \|_{W_2^m(\M\setminus B(\xi,R))} $$
 (with a similar estimate for $\chi_{\eta}$). 
 Thus we have
\begin{eqnarray*}
|A_{\xi,\eta}| 
&\le& 
C_{\LL} \CB \CZ
h^{d/2+m-2}( \|\chi_{\xi}\|_{W_2^m(\M\setminus B(\xi,R))}+ \|\chi_{\eta}\|_{W_2^m(\M\setminus B(\eta,R))})\\
&\le&
 2C_{\LL} \CB \CZ \Cen  h^{d/2+m-2} q^{d/2 -m} e^{-\nu R/h}.
\end{eqnarray*}
The lemma follows with $\Cstiff = 2C_{\LL} \CB \CZ \Cen \rho^{m-d/2}$.
\end{proof}
By considering row and column sums, we have that 
\begin{align}\label{eq:stiffnessupperbound}
\left\|\stiffness{\Xi} \right\|_{2 \to 2} 
 \le  C_A  h^{d-2}
\end{align}
holds with $C_A= \Cstiff  (1+\sum_{n=1}^{\infty} (n+2)^d e^{-\frac{\nu}{2\rho}n})$.

\begin{lemma}\label{lem:stiffness_inverse}
Under Assumptions 1,  for $\Xi\subset \M$, the stiffness matrix satisfies
$$\|\vec{A}_{\Xi}^{-1}\|_{\ell_2\to\ell_2}\le \CH h^{-d}$$
with a constant $\CH$ which  is  independent of $\Xi$.
\end{lemma}
\begin{proof}
 Coercivity ensures
$
|a( \sum_{\xi \in \Xi} v_{\xi} \chi_{\xi},\sum_{\xi \in \Xi} v_{\xi} \chi_{\xi} )|
	 \ge 
c_{0} \left\|\sum_{\xi \in \Xi} v_{\xi} \chi_{\xi}  \right\|^{2}_{W_{2}^{1}(\M)}
$,
	 and the metric equivalence $\|v\|_{W_2^1(\M)}^2 =  \|(1-\Delta)^{1/2}v\|_{L_2(\M)}$
gives
\begin{align*}
	 \vec{v} \cdot \stiffness{\Xi} \vec{v} 
	 \ge
	 c_0
	 \vec{v} \cdot 
	 \vec{L}
	 \vec{v},
	 \quad \text{where}
	 \quad
	 \vec{L}:=
	 \left( \langle  (1-\Delta)^{1/2}\chi_{\xi},(1-\Delta)^{1/2}\chi_{\eta}\rangle \right)_{\xi,\eta \in \Xi}
\end{align*}
is the stiffness matrix for the self-adjoint operator 
$
1-\Delta $. 
Because the spectrum of $1-\Delta$ is bounded below, i.e., $\sigma_{\mathrm{spec}}(1-\Delta)\subset [1,\infty)$,
we conclude that
$
	\vec{v} \cdot \vec{L} \vec{v} 
	\ge  
	\left\|\sum_{\xi \in \Xi} v_{\xi} \chi_{\xi}  \right\|^{2}_{L_{2}(\M)} 
	\ge  
	C_1 h^{d} \left\| \vec{v}\right\|^{2}_{\ell_{2}(\Xi)}
$
by the Riesz property (\ref{eq:riesz}).
Hence, overall we get 
\begin{align*}
	\|\stiffness{\Xi}\vec{v}\|_{\ell_2}
	\ge 
	 \frac{1}{\|\vec{v}\|_{\ell_{2}}}|\vec{v} \cdot \stiffness{\Xi} \vec{v}|
	 \ge %
	c
	h^{d} \left\| \vec{v}\right\|_{\ell_{2}}
\end{align*}
with $c = c_0 C_1$,
so $\|\stiffness{\Xi}^{-1}\|_{\ell_2\to\ell_2} \le \frac{1}{c_0C_1}h^{-d}$.
\end{proof}
Consequently, the $\ell_2$ condition number of $\stiffness{\Xi} $ is bounded by  a multiple of $h^{-2}$.
Please note that the bounds in Lemma \ref{lem:diag_decay} and Lemma \ref{lem:stiffness_inverse} 
do not assume the stiffness matrix to be symmetric, 
only these bounds are in some sense symmetric. 
Moreover, (\ref{eq:stiffnessupperbound}) is obtained by Riesz-Thorin interpolation and bounds on the row and column sums. Thus this bound also does not require the stiffness matrix to be symmetric.
\subsection{The diagonal of the stiffness matrix}

As a counterpart to the off-diagonal decay  given in Lemma \ref{lem:diag_decay}, we  can give the following
lower bounds on the diagonal entries.
\begin{lemma} \label{lem:stiffness_diagonal}
For an elliptic operator $\LL$ and a kernel satisfying Assumption 1, for a mesh ratio $\rho$,
there is $\Cdiag >0$ so that for any $\Xi\subset \M$ with $ h/q<\rho$, and any $\xi\in \Xi$,
we have
$$A_{\xi,\xi} = a(\chi_{\xi},\chi_{\xi}) \ge \Cdiag h^{d-2} .$$ 
\end{lemma}
\begin{proof}
By coercivity of $a$, it suffices to prove 
 that
$\|\nabla \chi_{\xi}\|_{L_2(B(\xi,h))}^2\gtrsim h^{d-2}.$

 \normalmarginpar
We begin by establishing a Poincar{\'e}-type inequality which is valid for smooth Lagrange functions.
To this end, consider $f: B(0,\mathrm{r}_{\M})\to \R$  obtained by the change of variable $f (x)= \chi_{\xi}(\exp_{\xi}(x))$.
Note that for $B\subset B(\xi,\mathrm{r}_{\M})$, we have
for any $0\le k\le m$, that
 $$
 \|f\|_{W_2^k(B)} \le
 \frac{1}{\Gamma_1}
 \|\chi_{\xi}\|_{W_2^k(\M)} \le 
  \frac{1}{\Gamma_1}
 (CZ
h^{m-k} )(\Cen \rho^{m-d/2} h^{d/2-m})
=C_1h^{d/2-k}.
 $$
by the zeros estimate, with $C_1 := \frac{1}{\Gamma_1}\rho^{m-d/2}  \Cen \CZ$, where the constants stem from (\ref{eq:expmap}).
 Now let $r:=\dist(\xi,\Xi\setminus \{\xi\})$, and define $F:B(0,1)\to \R$ by
$F:= f(r\cdot)$. Then, by a change of variable,
 $$\|F\|_{W_2^m(B(0,1))}^2 
 =  \sum_{k=0}^m |F|_{W_2^k(B(0,1))}^2 
 =
 \sum_{k=0}^m r^{2k-d} |f|_{W_2^k(B(0,r))}^2  
 \le C_2$$ 
 with  $C_2 := \rho^d C_1$, since $h/\rho\le q\le r\le h$.
 
 Because of the imbedding 
 $W_2^m(B(0,1))\subset C(\overline{B(0,1)})$ 
(which holds since $m>d/2$),
 the set $$K:=\left\{G\in W_2^m(B(0,1))\mid  \|G\|_{W_2^m(B(0,1))} \le C_2, \, G(0)=1, G(e_1)=0 \right\}$$
 is well defined, closed and convex, hence weakly compact, by Banach-Alaoglu. 
 
 The natural imbedding $\iota: W_2^m(B(0,1))\to W_2^1(B(0,1))$,
 is compact. We wish to show that $\iota(K)$ is a compact set.
 
 Because $\iota$ is a continuous linear map, it is continuous between the weak topologies
 of $W_2^m(B(0,1))$ and $W_2^1(B(0,1))$. Thus  $\iota(K)$ 
 is weakly compact in $W_2^1(B(0,1))$,
 and thus norm closed. Finally, because $\iota(K)$ is complete and totally bounded, 
 it is a compact subset in the norm topology of $W_2^1(B(0,1))$.
  
Consider the  (possibly zero) constant 
$c$ defined by 
$$
c:=\min_{G\in K}  \frac{\|\nabla G\|_{L_2(B(0,1))}}{\|G\|_{L_2(B(0,1))}} .
$$

The map  $I:W_2^1(B(0,1))\to \R:G\mapsto \frac{\|\nabla G\|_{L_2}}{\|G\|_{L_2}}$ 
is  continuous on the  complement of $0\in W_2^1(B(0,1))$ 
(as quotient of two continuous functions that do not vanish).
 In particular, it is continuous  and non-vanishing on $\iota(K)$, 
 so $c = \min_{G\in \iota(K)} I(G)>0$.
 Indeed, $I(G)>0$ for all $G\in \iota(K)$, since $G(0)=1$, $G(e_1)=0$ and $G\in C^1(B)$.

Note that in the above minimization problem, 
the condition $G(e_1)=0$ could be replaced by any other point on the unit sphere 
without changing  the value of $c$.
By rotation invariance of the $W_2^m(\R^d)$ norm, it follows  that $\|\nabla F\|_{L_2} \ge c\|F \|_{L_2}$. Finally,
employing the change of variables $r^{d-2}\|\nabla F\|_{L_2(B(0,1))}^2= \|\nabla f\|_{L_2(B(0,r))}^2$
and $\| F\|_{L_2(B(0,1))}^2= r^d\|f\|_{L_2(B(0,r))}^2$, we have
\begin{eqnarray*}
\|\nabla \chi_{\xi}\|_{L_2(B(\xi,r))}^2
& \ge & {\Gamma_1^2}r^{d-2}\|\nabla F\|_{L_2(B(0,1))}^2\\
&\ge& c^2 {\Gamma_1^2}r^{d-2}\| F\|_{L_2(B(0,1))}^2\\
&\ge& c^2\Gamma_1^2\rho^{-2} h^{-2} \| \chi_{\xi}\|_{L_2(B(\xi,r))}^2.
\end{eqnarray*}
The last line follows because $\chi_{\xi}$ is  H{\"o}lder continuous, so stays close to $1$ near $\xi$.
Specifically,  by (\ref{eq:holder}), for $\kappa := \left(2  \CH \right)^{-1/\epsilon}$ 
we have $\chi_{\xi}(x)>\frac12$ for all  $x\in B(\xi, \kappa h)$.
Thus 
$$\| \chi_{\xi}\|_{L_2(B(\xi,r))}^2 \ge \int_{B(\xi, \kappa h)} |\chi_{\xi} (x)|^2\dif x \ge \frac14 \alpha_{\M} (\kappa h)^d.$$
The lemma follows with constant $\Cdiag= \frac14c^2\Gamma_1^2\rho^{-2} \alpha_{\M} \kappa^d$.
\end{proof}

This brings us to the lower bound for diagonal entries of the stiffness matrix. Define the diagonal of $\vec{A}_{\Xi}$ as
$\stiffnessprecon{\Xi}:=\mathrm{diag}(\stiffness{\Xi}),$
and note that by Lemma \ref{lem:diag_decay} and Lemma \ref{lem:stiffness_diagonal}, 
\begin{equation*}
\kappa(\stiffnessprecon{\Xi}) =\frac{\max_{\xi} A_{\xi,\xi}}{\min_{\xi} A_{\xi,\xi}} \le \frac{\Cstiff}{\Cdiag}
\end{equation*}
is bounded above by a constant which depends only on the mesh ratio $\rho$
(and not on $\Xi$).

This permits us to find suitable damping constants $0<\theta<1$ so that 
$\stiffnessprecon{\Xi}$ dominates $\theta \stiffness{\Xi}$.
This drives
the success of the damped Jacobi method
considered in the next section.

\begin{lemma}\label{lem:smoothing_condition}
For an elliptic operator $\LL$, a kernel $\phi$ satisfying Assumption 1, and mesh ratio $\rho$, there is $\theta\in(0,1)$ so that for any point set 
$\Xi\subset \M$, 
$\theta \langle \stiffness{\Xi} \vec{v},\vec{v}\rangle\le \langle \stiffnessprecon{\Xi}\vec{v},\vec{v}\rangle$ for all $\vec{v} \in \R^{\Xi}$.
\end{lemma}
\begin{proof}
By (\ref{eq:stiffnessupperbound}), $ \langle \stiffness{\Xi} \vec{v},\vec{v}\rangle\le C_Ah^{d-2} \|v\|^2$, while  
$\Cdiag h^{d-2} \|v\|^2 \le \langle \stiffnessprecon{\Xi}\vec{v},\vec{x}\rangle$
follows from Lemma \ref{lem:stiffness_diagonal}. Thus the lemma holds for any $\theta$ in the interval $ (0,{\Cdiag}/{C_A}]$.
\end{proof}

\section{The  smoothing property}\label{sec:smoothing_prop}
In this section we define and study the \emph{smoothing operator} used in the multigrid method. 
We focus on the damped Jacobi method for the linear system $\stiffness{\Xi} \vec{u}^{\star}_{\Xi}=\vec{b}$,
where $ \stiffness{\Xi} \in \R^{|\Xi| \times |\Xi| }$ and  $\vec{b}\in \R^{|\Xi|}$.
For  a fixed damping parameter $0< \theta <1$,
 $\vec{u}^{(j)}$ is approximately computed via the iteration:
$$
\vec{u}^{(j+1)} = \theta \stiffnessprecon{\Xi}^{-1} \vec{b}+
( \id - \theta \stiffnessprecon{\Xi}^{-1} \stiffness{\Xi})\vec{u}^{(j)}, j\ge 0,$$
with  starting value $\vec{u}^{(0)} \in \R^{|\Xi|}$.
Define the affine map governing a single iteration as
\begin{align*}
	{J}_{\Xi}: \R^{\Xi }\times \R^{|\Xi|} \to \R^{|\Xi|}, \quad 
	{J}_{\Xi}(\vec{x},\vec{b})  =
	 \left(  \id - \theta \stiffnessprecon{\Xi}^{-1} \stiffness{\Xi}\right) \vec{u} + \theta \stiffnessprecon{\Xi}^{-1}\vec{b}.
\end{align*} 

This shows that the iteration converges if and only if 
iteration matrix
\begin{align}\label{eq:smoothingmat}
\smoothingmat{\Xi} := \id -\theta \stiffnessprecon{\Xi}^{-1}\stiffness{\Xi},
\end{align}
called the {\em  smoothing matrix},
is a contraction.

In the context of kernel approximation,
the corresponding  operator,
the {\em smoothing operator}, 
 $\smoothingop{\Xi}: V_{\Xi} \to V_{\Xi}$ 
is defined
as 
 $
 \smoothingop{\Xi} \sigma_{\Xi}^* 
 := 
 \sigma_{\Xi}^* \smoothingmat{\Xi}  .
$
 The success of the multigrid method relies on a {\em smoothing property}, which for our 
 purposes states that iterating $\smoothingop{\Xi} $ is  eventually contracting:
 $ \|\smoothingop{\Xi}^{\nu}\|_{ L_2(\M)\to \LL}  \le C h^{-1} o(\nu)$ as $\nu\to \infty$. This 
 smoothing property is demonstrated in section \ref{subsec:smoothing}.
Many of the results in the forthcoming section are formulated both in a matrix including $\smoothingmat{\Xi}$ form 
and an operator form $\smoothingop{\Xi}$. This resembles a bit a change of basis transformation.

 \subsection{$L_2$ stability of the smoothing operator}
 At this point, we are in a position to show that that iterating this operator is stable on $L_2(\M)$.
 To help analyze the matrix $\smoothingmat{\Xi}=\id-\theta \stiffnessprecon{\Xi}^{-1} \stiffness{\Xi}$, 
 we introduce the  inner product
$$
\langle \vec{u},\vec{v}\rangle_{\stiffnessprecon{\Xi}}
:= 
\langle \stiffnessprecon{\Xi}^{1/2} \vec{u},\stiffnessprecon{\Xi}^{1/2} \vec{v}\rangle_{\ell_2(\R^{\Xi})}.
$$
Since $\stiffnessprecon{\Xi}$ is diagonal, and   its diagonal entries are 
$\langle \chi_{\xi},\chi_{\xi}\rangle_{\LL} \sim h^{d-2}$, we have
the  norm equivalence $\|\vec{M}\|_{\stiffnessprecon{\Xi}\to \stiffnessprecon{\Xi}} \sim \|\vec{M}\|_{2\to 2}$. Specifically, we have
\begin{equation} \label{eq:matrix_norm_equivalence}
\frac{1}{C_{\stiffnessprecon{}}}\|\vec{M}\|_{2\to 2}
\le
\|\vec{M}\|_{\vec{B}\to \vec{B}} \le   
C_{\stiffnessprecon{}}\|\vec{M}\|_{2\to 2}
\end{equation}
with constant of equivalence 
$C_{\stiffnessprecon{}} 
=  
\kappa(\stiffnessprecon{\Xi}^{1/2})
=\frac{\max \sqrt{a(\chi_\xi,\chi_\xi)}}{\min \sqrt{a(\chi_\xi,\chi_\xi)}} 
\le \sqrt{\Cstiff/\Cdiag}$.

  \begin{lemma}\label{lem:non-expansive}
For the damping parameter $\theta$ in Lemma  \ref{lem:smoothing_condition} 
and $\smoothingmat{\Xi}$ as defined in \eqref{eq:smoothingmat}, there is $C>0$ so that
  for all $\nu\in\N$, $\vec{u} \in \R^{\Xi}$  and $u=\sigma_{\Xi}^{*} \vec{u}\in V_{\Xi}$,
  $$
  \|\smoothingmat{\Xi}^{\nu}\|_{2\to 2}
  \le 
  \kappa(\stiffnessprecon{\Xi}^{1/2}),
  \quad \text{and} \quad
  \|\smoothingop{\Xi}^{\nu} u\|_{L_2(\M)} 
   \le  \frac{C_2}{C_1}\kappa(\stiffnessprecon{\Xi}^{1/2})  \|u\|_{\crchange{L_{2}(\M)}}.$$
  \end{lemma}
  We note that this holds even when $a$ is non-symmetric.
  \begin{proof}
 The matrix 
  $\stiffnessprecon{\Xi}^{1/2} (\id-\theta \stiffnessprecon{\Xi}^{-1}\stiffness{\Xi}) \stiffnessprecon{\Xi}^{-1/2}$ has the same spectrum as 
  the matrix
  $\id-\theta \stiffnessprecon{\Xi}^{-1}\stiffness{\Xi}$. 
  Furthermore, because 
$0\le \langle \stiffnessprecon{\Xi}-\theta \stiffness{\Xi}\vec{x},\vec{x}\rangle \le  \langle \stiffnessprecon{\Xi}\vec{x},\vec{x}\rangle$,
  the spectral radius of $\smoothingmat{\Xi}$ is no greater than $1$. 
  Thus 
  $$
  \|\smoothingmat{\Xi}\|_{\stiffnessprecon{\Xi} \to \stiffnessprecon{\Xi}} 
  = 
  \|\stiffnessprecon{\Xi}^{1/2} (\id-\theta \stiffnessprecon{\Xi}^{-1}\stiffness{\Xi}) \stiffnessprecon{\Xi}^{-1/2}\|_{2\to2}
  \le 
  1.
  $$
  It follows that $\|\smoothingmat{\Xi}^{\nu}\|_{\stiffnessprecon{\Xi}\to \stiffnessprecon{\Xi}} \le 1$ as well, and the matrix norm equivalence
  (\ref{eq:matrix_norm_equivalence}) guarantees that $\|\smoothingmat{\Xi}^{\nu}\|_{2\to 2} \le \kappa(\stiffnessprecon{\Xi}^{1/2})$.
  The second inequality follows from the Riesz property (\ref{eq:riesz}).
  \end{proof}

\subsection{Smoothing properties}\label{subsec:smoothing}
For $v=\sigma_{\Xi}^* \vec{v}\in V_{\Xi}$, we have $[v]_{\LL} = \sqrt{a(v,v)} = \sqrt{ \langle \stiffness{\Xi}\vec{v},\vec{v}\rangle}$.
By applying Cauchy-Schwarz, the Riesz property  and Lemma \ref{lem:non-expansive}, the chain of inequalities
$$
   [\smoothingop{\Xi}^{\nu} v ]_{\LL} 
   \le 
   \|\stiffness{\Xi}\smoothingmat{\Xi}^{\nu} \vec{v}\|_{\ell_2}^{1/2}
   \|\smoothingmat{\Xi}^{\nu} \vec{v}\|_{\ell_2}^{1/2}
   \le 
   C h^{-d/4}   
   \|\stiffness{\Xi}\smoothingmat{\Xi}^{\nu} \vec{v}\|_{\ell_2}^{1/2}
   \|v\|_{L_2}^{1/2}
$$
holds. In the  case that $a$ is symmetric, we have the following smoothness property.
%
%
\begin{lemma}
\label{lem:symmetric_smoothing}
For a given $\rho>0$ there is a constant $C$ so that for any $\theta$ chosen as in Lemma \ref{lem:smoothing_condition}, 
$\smoothingmat{\Xi}$ 
as defined in \eqref{eq:smoothingmat}, 
and $\Xi\subset \M$ with mesh ratio 
$\rho(\Xi)\le \rho$, the damped Jacobi iteration has smoothing operator $\smoothingop{\Xi}\in L(V_{\Xi})$ which satisfies
$$\|\smoothingop{\Xi}^{\nu} v\|_{\LL} \le 
{C}{\theta}^{-1/2}  \sqrt{ \frac{1}{\nu+1} }
h^{-1}\|v\|_{L_2}.
$$
\end{lemma}
\begin{proof}
This is a result of \cite[Theorem 7.9]{reuskenLN}, which shows that 
$ 
\|\stiffness{\Xi}\smoothingmat{\Xi}^{\nu} \vec{v}\|_{\ell_2}
\le  
\frac{Ch^{d-2}}{\theta(\nu+1)}\|\vec{v}\|_{\ell_2}
$.
It follows that 
$\|\stiffness{\Xi}\smoothingmat{\Xi}^{\nu} \vec{v}\|_{\ell_2}\le \frac{Ch^{d/2-2}}{\theta(\nu+1)}\|{v}\|_{L_2}$, 
and the result holds by the above discussion.
\end{proof}

When $a$ is not symmetric, we have the  smoothness property.

 \begin{lemma}\label{lem:non-symmetric_smoothing}
 Let $\smoothingop{\Xi}:V_{\Xi}\to V_{\Xi}: v=\sum v_{\xi} \chi_{\xi}  \mapsto \sum (\smoothingmat{\Xi}\vec{v})_{\xi} \chi_{\xi}$. 
 For  $\theta$ as in Lemma \ref{lem:smoothing_condition} and $\smoothingmat{\Xi}$ 
 as defined in \eqref{eq:smoothingmat} we have
 $$[\smoothingop{\Xi}^{\nu} v]_{\LL} \le \frac{C}{\sqrt[4]{\nu}}h^{-1}\|v\|_{L_2}.$$
 \end{lemma}
\begin{proof}
This follows from \cite[Theorem 7.17]{reuskenLN}, and by techniques of the proof of Lemma \ref{lem:symmetric_smoothing}.
\end{proof}

 \section{The direct kernel multigrid method}\label{sec:direct_mgm}
  We are now in a position to consider the multigrid method applied to the kernel based Galerkin method
 we have described in the previous sections.

 \normalmarginpar
 \subsection{Setup: Grid transfer}
 We consider a nested sequence of point sets 
\begin{align*}
\Xi_0 \subset \Xi_1 \subset \dots \Xi_{\ell} \subset \Xi_{\ell+1} \subset \dots \subset  \Xi_{L} \subset \M
\end{align*}
and associated kernel spaces $V_{\Xi_{\ell}}$ as described in section \ref{sec:Kernel_Setup}.
We denote the Lagrange basis for each such space 
$	(\chi^{(\ell)}_{\xi} )_{ \xi \in \Xi_{\ell} }$,
and with it the accompanying analysis map $\sigma_{\ell}:=\sigma_{\Xi_{\ell}}$, 
synthesis map $\sigma_{\ell}^*:=\sigma_{\Xi_{\ell}}^*$ and stiffness matrix $\vec{A}_{\ell}:=\vec{A}_{\Xi_{\ell}}$.
Moreover, we assume that there are constants $0< \gamma_1\le \gamma_2 <1$ and $\rho\ge 1$ such that
\begin{align}\label{eq:gridconst}
q_{\Xi_{\ell}}\le  h_{\Xi_{\ell}} \le \rho q_{\Xi_{\ell}}, \quad \text{and} \quad \gamma_1 h_{\Xi_{\ell}}\le h_{\Xi_{\ell+1}}\le \gamma_2 h_{\Xi_{\ell}}.
\end{align}
Note at this point that $\rho$ is a universal constant. 
Hence $\rho$ does not depend on $\ell$ and thus the constants in deriving the smoothing property do not depend on $\ell$. 
Thus, we obtain $ n_{\ell} \sim h^{-d}_{\Xi_{\ell}}$, for constants see (\ref{eq:numberofpoints}).
We will assume that $L$ is the largest index that we will consider.

 In this section, we discuss {\em grid transfer}: specifically, the operators and matrices which provide communication 
 between finite dimensional
 kernel spaces. These include natural {\em prolongation} and {\em restriction} maps and their corresponding matrices.
We show how these can be used to relate Galerkin projectors $P_{\Xi_{\ell}}$ and stiffness matrices $\vec{A_{\ell}}$.

\paragraph{Prolongation and restriction}
Denote the Lagrange basis of $V_{\Xi_{\ell}}$ by $(\chi_{\xi}^{(\ell)})$,
and note that by containment $V_{\Xi_{\ell-1}}\subset V_{\Xi_{\ell}}$, it follows that
 $ \chi^{(\ell-1)}_{\xi}=\sum_{\eta \in \Xi_{\ell}} \beta_{\xi,\eta} \chi^{(\ell)}_{\eta} $
 holds for some matrix of coefficients $\beta_{\xi,\ell}$.
Furthermore, from the Lagrange property,  the identity
$
	\chi^{(\ell-1)}_{\xi}(\eta) 
	=\sum_{\zeta \in \Xi_{\ell}} \chi^{(\ell-1)}_{\xi}(\zeta)  \chi^{(\ell)}_{\zeta} (\eta) 
$
\change{holds} for any  $\eta \in \Xi_{\ell-1}$.
By uniqueness, we  deduce that $\beta_{\xi,\eta}=\chi^{(\ell-1)}_{\xi}(\eta)$,
and we have
\begin{align*}
	\sum_{\xi \in \Xi_{\ell-1}} c_{\xi}\chi^{(\ell-1)}_{\xi}
	= 
	\sum_{\xi \in \Xi_{\ell-1}}c_{\xi}\sum_{\eta \in \Xi_{\ell}} \chi^{(\ell-1)}_{\xi}(\eta)  \chi^{(\ell)}_{\eta} 
	=
	 \sum_{\eta \in \Xi_{\ell}} \sum_{\xi \in \Xi_{\ell-1}}c_{\xi} \chi^{(\ell-1)}_{\xi}(\eta) \chi^{(\ell)}_{\eta} .
\end{align*}
This yields that the natural injection $\prolop{\ell}: V_{\Xi_{\ell-1}}\to V_{\Xi_{\ell}}$,
called  the {\em prolongation map},  which is described by the rectangular matrix 
$$\prolmat{\ell}:= 
	\left(\chi^{(\ell-1)}_{\xi}(\eta) \right)_{\xi \in \Xi_{\ell-1}, \eta\in \Xi_{\ell}}
	= 
	(\sigma_{\ell}^*)^{-1}\sigma_{\ell-1}^*\ \in \R^{n_{\ell} \times n_{\ell-1}}.$$
It is worth noting that $ \I_{\ell-1}^{\ell}\sigma_{\ell-1}^*
= \sigma_{\ell-1}^*$, so we have
the identity
\begin{equation}\label{eq:synthesis_identity}
\sigma_{\ell}^* \prolmat{\ell}
= \sigma_{\ell-1}^*.
\end{equation}

The corresponding {\em restriction map }
$\resop{\ell}: V_{\Xi_{\ell}}\to V_{\Xi_{\ell-1}}$
is described by the transposed matrix $\resmat{\ell}=\left(\prolmat{\ell}\right)^T$.
In other words, it is defined as
$ \resop{\ell}\sigma_{\ell}^* =\sigma_{\ell-1}^* \left(\prolmat{\ell} \right)^T$.

Note that we can use $\resmat{\ell}=\left(\prolmat{\ell}\right)^T$ to relate analysis maps at different levels, since
we can take the $a$-adjoint of both sides of the equation
(\ref{eq:synthesis_identity})
to obtain the following useful identity:
\begin{equation}
\label{eq:analysis_identity}
\sigma_{\ell-1} = \left(\prolmat{\ell} \right)^T\sigma_{\ell}.
\end{equation}
Moreover, the prolongation is both bounded from above and below. This is a kernel based analogue for \cite[Eq. (64)]{reuskenLN}.
\begin{lemma}
Using the notation from above, there 
is a constant $C_{pro}\ge 1$
depending on $\gamma$, $\rho$, $\M$ and the constants in Assumption 2 so that 
\begin{align}
	\label{eq:prologantionmatrixbound}
	1&\le 
	  \left\|\prolmat{\ell}  \right\|_{\ell_2\left(\Xi_{\ell-1}\right)\to \ell_2\left(\Xi_{\ell}\right) } 
	\le C_{pro}
\end{align}
holds for all $\ell \ge 1$.
\end{lemma}
\begin{proof}
We begin by estimating the $\ell_1( \Xi_{\ell-1})\to \ell_1( \Xi_{\ell})$ and $\ell_{\infty}( \Xi_{\ell-1})\to \ell_{\infty}( \Xi_{\ell})$ norms of 
$\prolmat{\ell}$ by taking column and row sums, respectively. These estimates can be  made almost simultaneously, because
the $(\xi,\eta)$ entry of $\prolmat{\ell}$ satisfies the bound
$|\chi^{(\ell-1)}_{\xi}(\eta) |\le \Cpw \exp(-\nu \frac{\dist(\xi,\eta)}{h_{\ell-1}})$
by (\ref{eq:ptwise_decay}).

Let $A:=\sum_{n=1}^{\infty} (n+2)^d \exp(- \frac{\nu}{\rho}n)$ and $B:=\sum_{n=1}^{\infty} (n+2)^d \exp(- \frac{\nu\gamma}{\rho}n)$,
and note that both numbers depend on $\rho, \gamma$ and the exponential decay rate $\nu$ from (\ref{eq:ptwise_decay}).
Applying (\ref{eq:finite_sum}), gives
$$
\sum_{\xi\in\Xi_{\ell-1}} |\chi^{(\ell-1)}_{\xi}(\eta) |
\le
\Cpw
\left(1+\frac{\beta_{\M}}{\alpha_{\M}}A\right)
$$
and
$$
\sum_{\eta\in\Xi_{\ell}} |\chi^{(\ell-1)}_{\xi}(\eta) |
\le 
\Cpw
\left(1+\frac{\beta_{\M}}{\alpha_{\M}}B\right).
$$
Finally, interpolation gives the upper bound
 $$C_{pro}:= 
 \Cpw
  \sqrt{(1+\frac{\beta_{\M}}{\alpha_{\M}}B)(1+\frac{\beta_{\M}}{\alpha_{\M}}A)}.$$

By using the definition $\Upsilon_{\ell}:=\Xi_{\ell} \setminus\Xi_{\ell-1} $
we write $\vec{P}_a = (\chi_{\xi}(\eta))_{\xi,\eta\in\Xi_{\ell-1}}$ and 
$\vec{P}_b = (\chi_{\xi}(\zeta))_{\xi\in \Xi_{\ell-1},\zeta\in \Upsilon_{\ell}}$.
Thus, we have 
	$\left\| \prolmat{\ell}  \vec{c} \right\|^2_{2}=\|\vec{c}\|^{2}_{2}+\|\vec{P}_b\tilde{\vec{c}}\|^{2}_{2}\ge \|\vec{c}\|^{2}_{2}$,
which gives the lower bound.
\end{proof}

\subsection{multigrid iteration -- two level case}
We now describe the multigrid algorithm, which is a composition of smoothing operators, restriction, coarse grid correction,
prolongation, and then smoothing. 

We begin by considering the solution of $\stiffness{\ell}\vec{u}_{\ell}^{\star}=\vec{b}_\ell$, 
where $\stiffness{\ell}$ is the
stiffness matrix associated to $V_{\Xi_{\ell}}$,
and where
$\vec{b}_{\ell} = \sigma_{\ell}( u_{\ell}^{\star})$,
is the data obtained from the Galerkin solution
 $u_{\ell}^{\star} = \sigma^{\ast}_{\ell} (\vec{u}_{\ell}^{\star})\in  V_{\Xi_{\ell}}$.
 Naturally, $u_{\ell}^{\star}$ is unknown (its coefficients are the solution of the above problem),
 but
 we can compute the data 
 $\sigma_{\ell} {u}_{\ell}^{\star} $ via 
\begin{align*}	
\vec{b}_{\ell}&=
	\sigma_{\ell}(u_{\ell}^{\star}) \\
	&= 
	\left(a( u_{\ell}^{\star},\chi^{(\ell)}_{\xi_{1}}), \dots , a(u_{\ell}^{\star},\chi^{(\ell)}_{\xi_{n_{\ell}}})\right)^{T}\\
	&= \left(\langle f, \chi^{(\ell)}_{\xi_{1}}\rangle_{L_{2}(\M)}, \dots , \langle f,\chi^{(\ell)}_{\xi_{n_{\ell}}}\rangle_{L_{2}(\M)}\right)^{T} \!\!.
\end{align*}
In other words, it is obtained from the right hand side $f$.

\begin{algorithm}
\caption{$\mathrm{TGM}_{\ell}$ \newline Two-grid method  with post-smoothing in vectorial form, see \cite[Eq. (20)]{reuskenLN}}
\label{alg:mg_twogrid}
\begin{algorithmic}[0]
\REQUIRE  $\vec{u}^{\text{old}}_{\ell} \in \R^{n_{\ell}}$, right-hand-sides $\vec{b}_{\ell}=\sigma_{\ell} u^{\star}_{\ell} \in \R^{n_{\ell}}$
\ENSURE new approximation $\R^{n_{\ell}} \ni \vec{u}^{\text{new}}_{\ell}\gets \tgm_{\ell}(\vec{u}^{\text{old}}_{\ell},\vec{b}_{\ell} )$

\IF{$\ell=0$} 
	\STATE $\vec{u}^{\text{new}}_{0} \gets \vec{A}^{-1}_{0}\vec{b}_{0}$
	\COMMENT{for coarsest grid use direct solver}
\ELSE 
	\STATE $\vec{u}_{\ell} \gets {J}_{\ell}^{\nu_{1}}(\vec{u}^{\text{old}}_{\ell},\vec{b}_{\ell})$
	\COMMENT{$\nu_1$ steps pre-smoothing}
	\STATE  $\vec{d}_{\ell-1} \gets \resmat{\ell}\left( \vec{b}_{\ell}-\stiffness{\ell} \vec{u}_{\ell}\right)$ 
	\COMMENT{restrict residual to coarser grid}
	\STATE $\tilde{\vec{e}}_{\ell-1} \gets \vec{A}^{-1}_{\ell-1} \vec{d}_{\ell-1}$ 
	\COMMENT{solve coarse grid problem}
	\STATE $\vec{u}_{\ell} \gets \vec{u}_{\ell}+ \prolmat{\ell}\tilde{\vec{e}}_{\ell-1}$ 
	\COMMENT{update with coarse grid correction}
	\STATE $\vec{u}^{\text{new}}_{\ell} \gets {J}_{\ell}^{\nu_{2}}(\vec{u}_{\ell},\vec{b}_{\ell} )$
	\COMMENT{$\nu_2$ steps post-smoothing}
\ENDIF
\RETURN $\vec{u}^{\text{new}}_{\ell}$

\end{algorithmic} 
\end{algorithm}

The output $\vec{u}^\text{new}_{\ell}=\tgm_{\ell}(\vec{u}^{\text{old}}_{\ell},\vec{b}_{\ell} )$
of the two-level multigrid algorithm 
with initial input $\vec{u}^{\text{old}}_{\ell}$ is given by the rather complicated formula
 \begin{equation}
 \label{eq:full_algorithm}
 \vec{u}^\text{new}_{\ell} = 
 {J}^{\nu_2}_{\ell}
 \left( {J}^{\nu_1}_{\ell}(\vec{u}^{\text{old}}_{\ell},\vec{b}_{\ell})
 + 
 \prolmat{\ell}\left(\stiffness{\ell-1}\right)^{-1}\resmat{\ell}
 \left(\vec{b}_{\ell}-\stiffness{\ell} {J}^{\nu_1}_{\ell}(\vec{u}^{\text{old}}_{\ell},\vec{b}_{\ell})\right),\vec{b}_{\ell}\right).
 \end{equation}

Because $\vec{u}\mapsto \mathrm{TGM}_{\ell}(\vec{u},\vec{b})$ is consistent (with $\stiffness{\ell}\vec{u}=\vec{b}_{\ell}$)
the corresponding iteration matrix is
\begin{align}\label{eq:twogrimat}
\twogridmat{\ell} := 
\smoothingmat{\ell}^{\nu_2}
 \left(
  \id 
 -
  \prolmat{\ell}\left(\stiffness{\ell-1}\right)^{-1}\resmat{\ell}\stiffness{\ell}
\right)  \smoothingmat{\ell}^{\nu_1},
\end{align}
(this is calculated in  \cite[Eq. 48]{reuskenLN} as well as in \cite[Lemma  11.11]{hackbusch1994iterative}).
Thus, the error can be expressed as 
$ \vec{u}^\text{new}_{\ell} - \vec{u}_{\ell}^{\star} =\tgm_{\ell}(\vec{u}^{\text{old}}_{\ell}) - \vec{u}^{\star}_{\ell} = \twogridmat{\ell} (\vec{u}^{\text{old}}_{\ell}-\vec{u}^{\star}_{\ell})$.

The corresponding operator on $V_{\Xi_{\ell}}$
 is obtained by conjugating with $\sigma_{\ell}^*$. This gives the 
 error operator for the two level method $\twogridop{\ell} \sigma_{\ell}^{*}:= \sigma_{\ell}^{*}  \twogridmat{\ell} $.

It is worth noting that by (\ref{eq:synthesis_identity}) we have the equality
$
\sigma_{\ell}^{*} \prolmat{\ell}\left(\stiffness{\ell-1}\right)^{-1}\resmat{\ell}
\stiffness{\ell} 
=
\sigma_{\ell-1}^{*}\left(\stiffness{\ell-1}\right)^{-1}\resmat{\ell}
\stiffness{\ell}
$.
Using the identity $\stiffness{\ell}=\sigma_{\ell}\sigma_{\ell}^*$
followed by
(\ref{eq:analysis_identity})
and the identity  
$P_{\Xi_{\ell-1}}=\sigma^{\ast}_{\ell-1} \left( \stiffness{\ell-1} \right)^{-1}\sigma_{\ell-1}$
gives
\begin{equation}\label{eq:coarse_correction_identity}
\sigma_{\ell}^{*} \prolmat{\ell}\left(\stiffness{\ell-1}\right)^{-1}\resmat{\ell}\stiffness{\ell} 
=
P_{\Xi_{\ell-1}}
\sigma_{\ell}^*.
\end{equation}
It follows that 
$
\twogridop{\ell}
=
  \smoothingop{\ell}^{\nu_2} (\id_{V_{\Xi_{\ell-1}}} - P_{\Xi_{\ell-1}})\smoothingop{\ell}^{\nu_1} 
  $.

We are now in a position to show that the two level method is a contraction for  sufficiently large values of $\nu_1$.
  \begin{proposition}\label{prop:boundtwogridop}
  There is a constant $C$ so that for all $\ell$,
	$ \twogridop{\ell}$%
	satisfies the bound
	\begin{align*}
		\left\|  \twogridop{\ell}  \right\|_{L_{2}(\M)\to L_{2}(\M)} 
		\le 
		C_{\text{Prop.\ref{prop:boundtwogridop}}} g(\nu_1) 
		=
		C_{\text{Prop.\ref{prop:boundtwogridop}}}  \begin{cases} 
		   (2\nu_1 +1)^{-\frac{1}{2}}, & \text{symmetric }; \\ 
		    \nu_1^{-\frac{1}{4}}, & \text{non-symmetric }.
		 \end{cases}
	\end{align*}
 \end{proposition}
\begin{proof}
The following equality 
 $$\left\| \twogridop{\ell} \right\|_{L_2(\M)\to L_{2}(\M)}
=
\left\|  \smoothingop{\ell}^{\nu_2} (\id_{V_{\Xi_{\ell-1}}} - P_{\Xi_{\ell-1}})\smoothingop{\ell}^{\nu_1}  \right\|_{L_2(\M)\to L_{2}(\M)}$$
holds,
so Lemma \ref{lem:non-expansive} ensures
$$
\left\| 
\twogridop{\ell} \right\|_{L_2(\M)\to L_{2}(\M)}\le C \left\|  (\id_{V_{\Xi_{\ell-1}}} - P_{\Xi_{\ell-1}})\smoothingop{\ell}^{\nu_1} 
\right\|_{L_2(\M)\to L_{2}(\M)}.
$$
By Lemma \ref{lem:Nitsche},
$
\left\|  \id_{V_{\Xi_{\ell-1}}} - P_{\Xi_{\ell-1}}\right\|_{W_2^1(\M)\to L_2(\M)}\le C h_{\ell-1}$
holds, so 
it follows that
$$
\left\|  (\id_{V_{\Xi_{\ell-1}}} - P_{\Xi_{\ell-1}})\smoothingop{\ell}^{\nu_1}  \right\|_{L_2(\M)\to L_2(\M)}
\le  
C h_{\ell-1} \| \smoothingop{\ell}^{\nu_1}  \|_{L_2(\M)\to W_2^1(\M)}.
$$
By coercivity, this gives $\left\| \twogridop{\ell} \right\|_{L_2(\M)\to L_2(\M)}\le C h_{\ell-1} \| \smoothingop{\ell}^{\nu_1} v \|_{L_2(\M)\to \LL}$.
Finally, the result follows by applying the smoothing property: Lemma \ref{lem:symmetric_smoothing} in the symmetric case and Lemma \ref{lem:non-symmetric_smoothing} in the non-symmetric case.
 \end{proof}

 \begin{corollary}\label{cor:2g_error}
 Let $\theta$ \change{be} as in Lemma \ref{lem:smoothing_condition} 
 and let $\vec{u}^{\text{old}}_{\ell}\in \R^{n_{\ell}}$ be an initial guess, 
$u^{\text{old}}_{\ell} = \sigma_{\ell}^*(\vec{u}^{\text{old}}_{\ell})$,
 $\vec{u}_{\ell} = \tgm_{\ell}(\vec{u}^{\text{old}}_{\ell})$, 
  and  $u_{\ell} = \sigma_{\ell}^* \vec{u}_{\ell}$.
 If $a$ is symmetric, there is a constant $C$ independent of $\ell$ and $\theta$ so
 that 
 $$\|u^\text{new}_{\ell} - u_{\ell}^{\star}\|_{L_2(\M)}\le C\theta^{-1/2} \frac{1}{\sqrt{2\nu_1+1}}\|u^{\text{old}}_{\ell}-u_{\ell}^{\star}\|_{L_2(\M)}.$$
 If $a$ is not symmetric, 
  there is a constant $C$ independent of $\ell$ and $\theta$ so
 that 
 $$\|u^\text{new}_{\ell} -u_{\ell}^{\star} \|_{L_2(\M)}\le C \theta^{-1/2} \frac{1}{\sqrt[4]{\nu_1}}\|u^{\text{old}}_{\ell}-u_{\ell}^{\star}\|_{L_2(\M)}.$$
 \end{corollary}
 \begin{proof}
 This follows by applying Proposition \ref{prop:boundtwogridop}
 to $u^{\text{old}}_{\ell}-u_{\ell}^{\star}$.
 \end{proof}
 
 %
 %
 %
%
%
%
%
 \subsection{multigrid with $\tau$-cycle}
 In the two-grid method, the computational bottleneck remains the solution on the coarse grid. 
 Thus, there have been many approaches to recursively apply the multigrid philosophy in order to use a direct solver only on the coarsest grid. 
 A flexible algorithm is the so-called $\tau$-cycle. 
 Here $\tau=1$ stands for the $V$-cycle in multigrid methods and $\tau=2$ gives the $W$-cycle.
 
 Our results hold for $\tau\ge 2$.

\begin{algorithm}
\caption{$\mgm^{(\tau)}_{\ell}$\newline Multigrid method with $\tau$ cycle,  see \cite[Eq. (31)]{reuskenLN} }
\label{alg:mg_tau}
\begin{algorithmic}[0]
\REQUIRE  approximation: $\vec{u}^{\text{old}}_{\ell} \in \R^{n_{\ell}}$, right-hand-side $\vec{b}_{\ell} \in \R^{n_{\ell}}$, $\theta$
\ENSURE new approximation $\R^{n_{\ell}} \ni \vec{u}^\text{new}_{\ell}\gets \mgm^{(\tau)}_{\ell}(\vec{u}^{\text{old}}_{\ell},\vec{b}_{\ell} )$

\IF{$\ell=0$} 
	\STATE $\vec{u}^{\text{new}}_{0} \gets \vec{A}^{-1}_{0}\vec{b}_{0}$
	\COMMENT{for coarsest grid use direct solver}
\ELSE 
	\STATE $\vec{u}_{\ell} \gets {J}_{\ell}^{\nu_{1}}(\vec{u}^{\text{old}}_{\ell},\vec{b}_{\ell})$
	\COMMENT{$\nu_1$ steps pre-smoothing}
	\STATE  $\vec{d}_{\ell-1} \gets \resmat{\ell}\left( \vec{b}_{\ell}-\stiffness{\ell} \vec{u}_{\ell}\right)$ 
	\COMMENT{restrict residual to coarser grid}
	\STATE $\vec{e}^{(0)}_{\ell-1} \gets 0$ 
	\COMMENT{initialize}
	\FOR{$i=0$ \TO $\tau$} 
		\STATE $\vec{e}^{(i)}_{\ell-1} \gets  \mgm^{(\tau)}_{\ell-1}(\vec{e}^{i-1}_{\ell-1},\vec{d}_{\ell-1})$ 
		\COMMENT{recursive call on $\Xi_{\ell-1}$}
	\ENDFOR
	\STATE $\vec{u}_{\ell} \gets \vec{u}_{\ell} + \prolmat{\ell} \vec{e}^{(\tau)}_{\ell-1}$
	\COMMENT{update with coarse grid correction}
	\STATE $\vec{u}^\text{new}_{\ell} \gets J_{\ell}^{\nu_{2}}(\vec{u}_{\ell} ,\vec{b}_{\ell})$
	\COMMENT{$\nu_2$ steps post-smoothing}
\ENDIF
\RETURN $\vec{u}^\text{new}_{\ell}$

\end{algorithmic} 
\end{algorithm}

Before proving  our main theorem, we need a statement from elementary real analysis.
\begin{lemma}\label{lem:recursivetau}
For any real numbers $\alpha,\beta,\gamma,\tau$ which satisfy
$0<\gamma<1$,
$\tau\ge 2$, 
$\beta>1/\tau$  and $\alpha<\min\left\{ \frac{\tau-1}{\tau} (\beta \tau)^{-\frac{1}{\tau-1}}, \frac{\tau-1}{\tau} \gamma\right\}$,
if the  
sequence $(x_n)_{n \in \N_0}$ satisfies the conditions
\begin{align*}
	x_0=0, \quad x_{n+1}\le \alpha +\beta \bigl(x_n\bigr)^{\tau}
\end{align*}
then
$x_n \le \gamma$ for all 
 $n\ge 0$.
\end{lemma}
\begin{proof}
This follows by elementary calculations, as in \cite[Lemma 6.15]{johnLN}.
\end{proof}
 Using \cite[Theorem 7.1]{reuskenLN}, we obtain for the iteration matrix of the Algorithm \ref{alg:mg_tau} the recursive (in the level)  form
 \begin{align*}
 	\multigridmat{\ell}=
	\begin{cases}
		\vec{0 },& \ell=0 \\
		\twogridmat{\ell}+\smoothingmat{\ell}^{\nu_2} \prolmat{\ell}  (\multigridmat{\ell-1})^{\tau} \vec{A}_{\ell-1}^{-1}\resmat{\ell} \stiffness{\ell} \smoothingmat{\ell}^{\nu_1}, & \ell\in \N.
	\end{cases} 
 \end{align*}
 Again, we define the corresponding operator via $\multigridop{\ell}{\tau} := \sigma_{\ell}^{*}  \multigridmat{\ell} ( \sigma_{\ell}^{*})^{-1}$.
 \begin{theorem}\label{thm:main}
For every $\gamma \in (0,1)$, there is a $\nu^{\star}:= \arg\min_{\nu \in \N} \{ \nu  \in \N  \ :  \  C_{\text{Prop.\ref{prop:boundtwogridop}}} g(\nu) \le \min\left\{ \frac{\tau-1}{\tau} (\beta_{\text{Thm.\ref{thm:main}}}  \tau)^{-\frac{1}{\tau-1}}, \frac{\tau-1}{\tau} \gamma\right\}\} $. For all $\nu_1 \ge \nu^{\star}$ we have 
\begin{align}
 \left\| \multigridop{\ell}{\tau} \right\|_{L_{2}(\M)  \to L_{2}(\M)} \le \gamma.
\end{align}
\end{theorem}
 \begin{proof}
 	Here, we follow basically \cite[Proofof Theorem 7.20]{reuskenLN}. Let $\vec{v} \in \R^{n_{\ell}}$ arbitrary. 
	For $v =\sigma^{*}_{\ell} \vec{v}$, we obtain for $\ell \in \N$,
\begin{eqnarray*}
	\left\| \multigridop{\ell}{\tau} \right\|_{L_{2}(\M)\to L_{2}(\M) } 
		&\le &
	\left\| \twogridop{\ell} \right\|_{L_{2}(\M)\to L_{2}(\M)} \\
	&&+ 
	\left\|  
	    \smoothingop{\ell}^{\nu_2} \sigma^{*}_{\ell}\prolmat{\ell}  
	    (\multigridmat{\ell-1})^{\tau} \vec{A}_{\ell-1}^{-1}\resmat{\ell} 
	    \stiffness{\ell}(\sigma^{*}_{\ell})^{-1} \smoothingop{\ell}^{\nu_1}
	 \right\|_{L_{2}(\M)\to L_{2}(\M)} \\
		&\le &
	 \left\| \twogridop{\ell} \right\|_{L_{2}(\M)\to L_{2}(\M)}\\
	&& + 
	 C 
	 \left\|\sigma^{*}_{\ell} \prolmat{\ell}  
	 (\multigridmat{\ell-1})^{\tau} \vec{A}_{\ell-1}^{-1}
	 \resmat{\ell} \stiffness{\ell}  (\sigma^{*}_{\ell})^{-1} \smoothingop{\ell}^{\nu_1}
	 \right\|_{L_{2}(\M) }
\end{eqnarray*}
by Proposition \ref{prop:boundtwogridop} and Lemma \ref{lem:non-expansive}.
We treat 	the second term   with (\ref{eq:synthesis_identity}) and (\ref{eq:coarse_correction_identity})
to obtain
$$
\sigma^{*}_{\ell} \prolmat{\ell} ( \multigridmat{\ell-1})^{\tau} \vec{A}_{\ell-1}^{-1}\resmat{\ell} \stiffness{\ell} 
(\sigma^{*}_{\ell})^{-1} \smoothingop{\ell}^{\nu_1}
=
 \multigridop{\ell-1}{\tau}^{\tau}P_{\Xi_{\ell-1}} \smoothingop{\ell}^{\nu_1}  .
 $$
This leaves 
		\begin{align*}
		\left\| \multigridop{\ell}{\tau} v \right\|_{L_{2}(\M) } 
		&\le 
		  \left\| \twogridop{\ell} \right\|_{L_{2}(\M)\to L_{2}(\M)}\\
		  &\phantom{\le} + C \left\| \multigridop{\ell-1}{\tau}^{\tau} \right\|_{L_{2}(\M)\to L_{2}(\M)}
		 \left\| P_{\Xi_{\ell-1}}\smoothingop{\ell}^{\nu_1} \right\|_{L_{2}(\M)\to L_{2}(\M)}.
	\end{align*}
 The last factor can be bounded by writing $P_{\Xi_{\ell-1}}=\mathrm{id} -(\mathrm{id}-P_{\Xi_{\ell-1}})$ followed by the triangle inequality.
Lemma \ref{lem:non-expansive}  bounds $\|\smoothingop{\ell}^{\nu_1}\|_{L_2(\M)\to L_2(\M)}$, 
 while Proposition \ref{prop:boundtwogridop} (with $\nu_2=0$) bounds
  $\|(\mathrm{id}-P_{\Xi_{\ell-1}})\smoothingop{\ell}^{\nu_1}\|_{L_2(\M)\to L_2(\M)}$.
 Thus, we end up with a bound 
 \begin{align*}
 	\left\| \multigridop{\ell}{\tau} \right\|_{L_{2}(\M)  \to L_{2}(\M)} 
	\le 
	 \left\| \twogridop{\ell} \right\|_{L_{2}(\M)\to L_{2}(\M)}
	    + 
	    C \left\| \multigridop{\ell-1}{\tau} \right\|^\tau_{L_{2} \to L_{2}} ,
 \end{align*}
 which has the form required by Lemma \ref{lem:recursivetau},
 with $x_{\ell} = \left\| \multigridop{\ell}{\tau} \right\|_{L_{2}  \to L_{2}} $,
 $ \beta_{\text{Thm.\ref{thm:main}}} :=C$
 and $\alpha=  \left\| \twogridop{\ell} \right\|_{L_{2}(\M)\to L_{2}(\M)}$.
 The condition $\nu_1 \ge \nu^{\star}$ ensures the bound $\alpha \le \min\left\{ \frac{\tau-1}{\tau} (\beta_{\text{Thm.\ref{thm:main}}}  \tau)^{-\frac{1}{\tau-1}}, \frac{\tau-1}{\tau} \gamma\right\}$.
Thus 
\begin{align*}
 \left\| \multigridop{\ell}{\tau} \right\|_{L_{2}(\M)  \to L_{2}(\M)} \le \gamma 
 \end{align*}
 holds  by Lemma \ref{lem:recursivetau}, and the theorem follows.
 \end{proof}

\begin{remark}
At the finest level, the kernel-based Galerkin problem $\vec{A}_{L}\vec{x}=\vec{b}_{L}$, can be solved stably to any precision $\epsilon_{\max}$,
 by iterating the contraction matrix $\multigridmat{L}{\tau}$ .
Select $\gamma<1$ and fix $\nu_1$ so that Theorem \ref{thm:main} holds.
 Letting $\vec{u}^{(k+1)} =  \mgm^{(\tau)}_{L}(\vec{u}^{(k)},\vec{b}_{\ell} )$ gives
 $\|\vec{u}^*-\vec{u}^{(k)}\|_{\ell_2} \le \gamma^k\|\vec{u}^*-\vec{u}^{(0)}\|_{\ell_2} $.  
 If $k$ is the least integer satisfying  $\gamma^k\|\vec{u}^*-\vec{u}^{(0)}\|<\epsilon_{\max}$, 
 then $k\sim \frac{1}{\log \gamma} \log\left(\frac{\epsilon_{\max}}{\|\vec{u}^* -\vec{u}^{(0)}\|}\right)$.
 
 We note that 
 $\|\vec{u}^*-\vec{u}^{(k)}\|_{\vec{A}_L}
 \sim 
 \|\sigma_L^*\vec{u}^*-\sigma_L^*\vec{u}^{(k)}\|_{W_2^1} 
 \le 
 \CB h^{d/2-1} \|\vec{u}^*-\vec{u}^{(k)}\|_{\ell_2},$
 and since $d\ge2$, achieving $\|\vec{u}^*-\vec{u}^{(k)}\|_{\vec{A}_L}<\epsilon_{\max}$ 
 also requires only a fixed number of iterations. 
 This shows (\ref{eq:mgm_iteration_count}).
 \end{remark}

%
%
%
%
\section{The perturbed multigrid method}\label{sec:perturb}
In this section, we consider a modified problem 
$$\stiffnesspert{L}\check{\vec{u}}^{\star}_{L}=\vec{b}_{L},$$ where $\stiffnesspert{L}$ is close to $\stiffness{L}$.
The perturbed multigrid method will produce an approximate solution $\check{\vec{u}}^{(k)}_{L}$ to $\vec{u}^{\star}_{L}$  which satisfies 
$\| \check{\vec{u}}^{(k)}_{L}- \vec{u}^{\star}_{L}\|\le\| \check{\vec{u}}^{(k)}_{L} -  \check{\vec{u}}^{\star}_{L} \|+ \|\stiffnesspert{L}^{-1}-\stiffness{L}^{-1}\|\|\vec{b}_{L}\|$.
Thus, for the true solution $u$ to (\ref{eq:weakform})
and the Galerkin solution $u^{\star}_L = P_{\Xi_L}u = \sigma_{L}^* \vec{u}^{\star}_{L}$, we have 
$$\|\sigma^* \check{\vec{u}}^{(k)}_{L}- u\|_{W_2^k(\M)} \le \|(1-P_{\Xi_L})u\|_{W_2^k(\M)} + C h_L^{d/2-k} \|\vec{\check{u}}_{L}^{(k)} - \vec{\check{u}}_{L}^{\star}\|_{\ell_2}.$$
which can be made as close to $ \|(1-P_{\Xi_L})u\|_{L_2} $ as desired by controlling the perturbation  $\|\stiffnesspert{L}- \stiffness{L}\|_{2\to 2}$
and the error from the multigrid approximation $\|\vec{\check{u}}_{L}^{(k)} - \vec{\check{u}}_{L}^{\star}\|$.

Such systems may occur for a number of reasons: using localized  Lagrange basis functions (as in \cite{NRW}), 
truncating a series expansion of the kernel (as in \cite{collins2021kernel}), 
or by using quadrature to approximate the stiffness matrix (both \cite{NRW} and \cite{collins2021kernel}).
In the next section, we will apply this by truncating the original stiffness matrix to employ only banded matrices and thereby 
enjoy a computational speed up.

\paragraph{Perturbed multigrid method}
The perturbed multigrid method replaces matrices $ \stiffness{\ell}$,  $\prolmat{\ell}$ and $\resmat{\ell}$ appearing in 
Algorithms \ref{alg:mg_twogrid} and \ref{alg:mg_tau} with 
matrices $\stiffnesspert{\ell}$, $\prolpert{\ell}$ and $\respert{\ell}$.
We assume that for each $\ell$ there exists $0<\epsilon_{\ell}$ so that 
$$
 \|\stiffnesspert{\ell}- \stiffness{\ell}\|_{\ell_2\to \ell_2}, \|\prolpert{\ell}- \prolmat{\ell}\|_{\ell_2\to \ell_2}, \|\respert{\ell}- \resmat{\ell}\|_{\ell_2\to \ell_2} 
 \le \epsilon_{\ell}.
 $$
 In this set up $\epsilon_{\ell}$ may change per level.\footnote{Which could be the case, e.g., if $\stiffnesspert{\ell}$ involved a $\Xi_{\ell}$ dependent
 quadrature scheme, or was obtained by bandlimiting (as we will do in the next section)} 
We assume $\epsilon_{\ell}$ is  small enough that
$$
\|\stiffnesspert{\ell}\|_{2\to 2} \le 2C_Ah_{\ell}^{d-2},
\qquad
\|\prolpert{\ell}\|_{2\to 2}<2C_{pro}  
\quad
\text{ and }
\quad
\|\respert{\ell}\|_{2\to 2}<2C_{pro}.
$$
It then follows from standard arguments that
$$
 \|\stiffnesspert{\ell}^{-1}\|_{2\to 2} \le 2 \CH h_{\ell}^{-d}
 \quad
 \text{and}
 \quad
\|\stiffnesspert{\ell}^{-1}- \stiffness{\ell}^{-1}\|_2 
\le 
2(\CH)^2 h_{\ell}^{-2d}\epsilon_{\ell}.
$$
 Because of the entry-wise error $ | (\stiffnesspert{\ell})_{\xi,\xi} -  (\stiffness{\ell})_{\xi,\xi}|\le \epsilon_{\ell}$, 
 we also have that  the diagonal matrix $\preconpert{\ell}=\diag(\stiffnesspert{\ell})$
 satisfies 
 $$
 \kappa(\preconpert{\ell}) 
 \le  
 \frac{1+\epsilon_{\ell}}{1-\epsilon_{\ell}} \kappa(\stiffnessprecon{\ell})
 \le
 2\frac{\Cstiff}{\Cdiag}.
 $$
 Therefore,  there is $\theta$ so that for all $\ell$ and  all $\vec{x} \in \R^{n_{\ell}}$, 
$\theta \langle \stiffnesspert{\ell} \vec{x},\vec{x}\rangle\le \langle \preconpert{\ell}  \vec{x},\vec{x}\rangle$ holds.
This permits us to consider the Jacobi iteration applied to the perturbed linear system
$
\stiffnesspert{\ell}\vec{x}_{\ell}=\vec{b},
$
which yields  $\check{J}_{\ell}: \R^{n_{\ell} }\times \R^{n_{\ell}} \to \R^{n_{\ell}}$ defined by
\begin{align}\label{eq:smoothingmatpert}
	\check{J}_{\ell}(\vec{x},\vec{b})  =
	\smoothingmatpert{\ell} \vec{x} + \theta \preconpert{\ell}^{-1}\vec{b},
	\quad \text{where} \quad
	\smoothingmatpert{\ell} := \id -\theta \preconpert{\ell}^{-1}\stiffnesspert{\ell}.
	 \end{align}
 Since $ \smoothingmat{\ell} -\smoothingmatpert{\ell}  =  \theta \Bigl( \preconpert{\ell}^{-1}(\stiffnesspert{\ell}-\stiffness{\ell}) +( \preconpert{\ell}^{-1}- \stiffnessprecon{\ell}^{-1 } )\stiffness{\ell}\Bigr)$, we can estimate the error between smoothing matrices as
 \begin{equation}
 \label{eq:smoothingapprox}
 \| \smoothingmat{\ell} -\smoothingmatpert{\ell}  \|_{2\to 2} 
 \le 
  \frac{3\theta C_A}{(\Cdiag)^2 }  h_{\ell}^{2-d} \epsilon_{\ell}.
 \end{equation}
Because $\theta \langle \stiffnesspert{\ell} \vec{x},\vec{x}\rangle\le \langle \preconpert{\ell}  \vec{x},\vec{x}\rangle$ it follows
that for all $n$,
\begin{equation}\label{eq:perturb_nonexpansive}
\|\smoothingmatpert{\ell}^n\|_{2\to 2}\le\sqrt{ 2\frac{\Cstiff}{\Cdiag} }. 
\end{equation}
This also yields the following Lemma.
  \begin{lemma}\label{lem:smoothiterbound}
  For $M \in \N$, we get the bound
  	\begin{align*}
	 \left\|\smoothingmat{\ell}^{M}-\smoothingmatpert{\ell}^{M} \right\|_{2\to 2}\le  
	 2 M  \frac{\Cstiff}{\Cdiag}\left\| \smoothingmat{\ell}-\smoothingmattrunc{\ell}{r_{\Xi}} \right\|_{2 \to 2}.
	 \end{align*}
	\end{lemma}
  \begin{proof}
  By telescoping, we have $\smoothingmat{\ell}^{M}-\smoothingmatpert{\ell}^{M} 
  = \sum_{j=0}^{M-1} \smoothingmat{\ell}^{M-1-j}\left( \smoothingmat{\ell} -\smoothingmatpert{\ell}\right)\smoothingmatpert{\ell}^{j}$.
The inequality
\begin{align*}
\|\smoothingmat{\ell}^{M}-\smoothingmatpert{\ell}^{M}\|\le	
\|\smoothingmat{\ell} -\smoothingmatpert{\ell}\|
\sum_{m=0}^{M-1}\left\| \smoothingmat{\Xi}^m \right\|_{2 \to 2}\left\| \smoothingmatpert{\ell}^{M-1-m} \right\|_{2 \to 2}.
\end{align*}
follows from norm properties, and  the result follows by applying (\ref{eq:perturb_nonexpansive}) and Lemma \ref{lem:non-expansive}.
\end{proof}
This lemma can be combined with the estimate (\ref{eq:smoothingapprox})
to obtain 
\begin{align}\label{eq:smooth_pert}
 \left\|\smoothingmat{\ell}^{M}-\smoothingmatpert{\ell}^{M} \right\|_{2\to 2}
 \le  
 6\theta \frac{\Cstiff C_A}{(\Cdiag)^3}M  h_{\ell}^{2-d} \epsilon_{\ell}.
\end{align}
%
 \paragraph{Perturbed two grid method}
Now, we consider the perturbed version of the two step algorithm.
We aim to apply the two-grid method with only truncated matrices to the problem
\begin{align}\label{eq:trruncatedsystem}
\stiffnesspert{\ell} \check{\vec{u}}^{\star}_{\ell}=\vec{b}_{\ell}=\sigma_{\ell} u^{\star}_{\ell}.
\end{align}
Applying the two-grid method 
 to 
 (\ref{eq:trruncatedsystem}) gives
$
	\check{\vec{u}}^{\star}_{\ell}-\check{\vec{u}}^{\text{new}}_{\ell}= 
	\twogridmatpert{\ell} \left(\check{\vec{u}}^{\star}_{\ell}-\check{\vec{u}}^{\text{old}}_{\ell} \right),
$
where the two grid iteration matrix is
\begin{align}\label{eq:twogrimattrunc}
 \twogridmatpert{\ell}   := 
\smoothingmatpert{\ell}^{\nu_2}
 \left(\id -  \prolpert{\ell}  \left(\stiffnesspert{\ell-1}\right)^{-1}  
 \respert{\ell} \stiffnesspert{\ell}
\right)  \smoothingmatpert{\ell}^{\nu_1}.
\end{align}

\begin{lemma}\label{lem:2step_pert}
If $ \epsilon_{\ell} \le h_{\ell}^{d+2}$ holds for all $\ell\le L$,
then
\begin{align*}
\| \twogridmat{\ell} -\twogridmatpert{\ell}  \|_{2\to 2}\le 
C  (\nu_2 +\nu_1 + h_{\ell-1}^{-4}) h_{\ell-1}^{-(d+2)} \epsilon_{\ell-1}.
\end{align*}
\end{lemma}
\begin{remark}\label{remark:telescoping}
A basic idea, used throughout this section, is the following result: If $\max(\|M_j\|,\|\check{M}_j\|) \le C_j$, then 
\begin{equation*}
\left\| \prod_{j=1}^n M_j - \prod_{j=1}^n \check{M}_j \right\| \le \left(\prod_{j=1}^n C_j\right)\left(\sum_{j=1}^n \frac{\|M_j-\check{M}_j\|}{C_j}\right).
\end{equation*}
\end{remark}

\begin{proof}[Proof of Lemma \ref{lem:2step_pert}]
Consider 
\begin{eqnarray*}
E_1
&:=&
\left( \id - \prolpert{\ell} \stiffnesspert{\ell-1}^{-1}  \respert{\ell} \stiffnesspert{\ell}\right)
-
\left( \id -\prolmat{\ell}\stiffness{\ell-1}^{-1}\resmat{\ell}\stiffness{\ell} \right)\\
&=&
\prolmat{\ell}\stiffness{\ell-1}^{-1}\resmat{\ell}\stiffness{\ell}-
\prolpert{\ell} \stiffnesspert{\ell-1}^{-1}  \respert{\ell} \stiffnesspert{\ell}.
\end{eqnarray*}
By Remark \ref{remark:telescoping}, we have
	\begin{align*} 		
	\left\|E_1 \right\|_{2\to 2}
	& \le  C h_{\ell-1}^{-d}h_{\ell}^{d-2}
	         \left(  2\frac{\epsilon_{\ell}}{2C_{pro}}	          
	          + \frac{  2( \CH)^2
	             \epsilon_{\ell-1} h_{\ell-1}^{-2d}}{2 \CH
	             h_{\ell-1}^{-d}}
		+  \frac{\epsilon_{\ell}}{  2C_A h_{\ell}^{d-2} }\right)\\
		&\le C h_{\ell-1}^{-d-2} \epsilon_{\ell-1}.
	\end{align*}
	Now, we consider the difference
	\begin{align*} 	E_{2}&:=\twogridmat{\ell}  -\twogridmatpert{\ell}\\
	&=
	\smoothingmat{\ell}^{\nu_2} 
	\left(\id-\prolmat{\ell} \stiffness{\ell-1}^{-1}  \resmat{\ell}\stiffness{\ell} \right) 
	 \smoothingmat{\ell}^{\nu_1}-
	  \smoothingmatpert{\ell}^{\nu_2}
          \left(\id -  \prolpert{\ell}  \stiffnesspert{\ell-1}^{-1}   \respert{\ell}  \stiffnesspert{\ell}\right)  
          \smoothingmatpert{\ell}^{\nu_1}.
          \end{align*}
Because $ \|\id -  \prolpert{\ell}  \stiffnesspert{\ell-1}^{-1}   \respert{\ell}  \stiffnesspert{\ell}\|\le C h_{\ell}^{-2}$ and $\|\id -  \prolpert{\ell}  \stiffnesspert{\ell-1}^{-1}   \respert{\ell}  \stiffnesspert{\ell}\|\le C h_{\ell}^{-2}$,
the lemma follows by using Remark \ref{remark:telescoping}
with  (\ref{eq:smoothingapprox}), (\ref{eq:smooth_pert}), Lemma \ref{lem:smoothiterbound}, and the above estimate of $\|E_1\|$.
\end{proof}
%
%
%
\paragraph{Perturbed $\tau$-cycle}
As in the two-step case we consider the multigrid method also for the truncated system in 
(\ref{eq:trruncatedsystem}), i.e., 
$\stiffnesspert{\ell} \check{\vec{u}}^{\star}_{\ell}=\vec{b}_{\ell}=\sigma_{\ell} u^{\star}_{\ell}$.
The multigrid iteration matrix is
\begin{align*}
	\multigridmatpert{\ell}=
	\begin{cases}
		\vec{0 },& \ell=0 \\
		\twogridmatpert{\ell}
		+\smoothingmatpert{\ell}^{\nu_2}   \prolpert{\ell} 
		\multigridmatpert{\ell-1}^{\tau}    \stiffnesspert{\ell-1}^{-1} 
		   \respert{\ell}     \stiffnesspert{\ell}
		\smoothingmatpert{\ell}^{\nu_1}, & 1\le \ell\le L .
	\end{cases} 
\end{align*}
From this we define the operator $\multigridoppert{\ell} := \sigma_{\ell}^{*}  \multigridmatpert{\ell} ( \sigma_{\ell}^{*})^{-1}$.
%
\begin{theorem}\label{thm:theorem_pert}
For  any $0<\gamma<1$, 
there exist constants $C_1$, $C_2$ and $C_4$ and $\nu^{\star}\in \N$ such that $C_{\text{Prop.\ref{prop:boundtwogridop}}} g(\nu^{\star}) \le C_4\min\left\{ \frac{\tau-1}{\tau} (\beta_{\text{Thm.\ref{thm:main}}} \tau)^{-\frac{1}{\tau-1}}, \frac{\tau-1}{\tau} \gamma \right\}$. 
For $\nu_1 \ge \nu^{\star}$ choose $\epsilon_{\ell}$ such that $\epsilon_{\ell} h_{\ell}^{-(d+2)}(h_{\ell}^{-4}+\nu_1+\nu_2)< C^{-1}_{\text{Lem.\ref{lem:2step_pert}}}\min(C_1,\gamma C_2)$ for $0\le \ell \le L$ for all $\ell$ and if $h_0 \le \left(C^{-1}_{\text{Lem.\ref{lem:2step_pert}}}\min(C_1,\gamma C_2)\right)^{-1/4}$, then
\begin{align*}
	\left\|\multigridoppert{\ell} \right\|_{L_{2}(\M) \to L_{2}(\M)}\le \gamma.
\end{align*}
\end{theorem}
\begin{proof}
As in the proof of Theorem \ref{thm:main}, we make the estimate 
\begin{equation*}
\|\multigridoppert{\ell}  \| \le
\|\twogridoppert{\ell}   \| +
 \|\smoothingoppert{\ell}^{\nu_2} \sigma_{\ell}^*  \prolpert{\ell} 
		\multigridmatpert{\ell-1}^{\tau}    \stiffnesspert{\ell-1}^{-1} 
		   \respert{\ell}     \stiffnesspert{\ell}(\sigma_{\ell}^* )^{-1}
		\smoothingoppert{\ell}^{\nu_1} \|.
\end{equation*}		
Then (\ref{eq:perturb_nonexpansive}) ensures that 
$
\| \sigma_{\ell}^*\prolpert{\ell} 
		(\sigma_{\ell-1}^* )^{-1}		
		\multigridoppert{\ell-1}^{\tau}    
		\sigma_{\ell-1}^* 
		\stiffnesspert{\ell-1}^{-1} 
		   \respert{\ell}     \stiffnesspert{\ell}
		(\sigma_{\ell}^* )^{-1}
		\smoothingoppert{\ell}^{\nu_1}
\|
$	
controls the  second expression. 
Considering the difference
 $$E := \sigma_{\ell}^*\prolpert{\ell} 	
		\multigridmatpert{\ell-1}^{\tau}    
		\stiffnesspert{\ell-1}^{-1} 
		   \respert{\ell}     \stiffnesspert{\ell}
		(\sigma_{\ell}^* )^{-1}
		\smoothingoppert{\ell}^{\nu_1}
		-
\sigma_{\ell}^* \prolmat{\ell} 		
		\multigridmatpert{\ell-1}^{\tau}    
		\stiffness{\ell-1}^{-1} 
		   \resmat{\ell}  
		      \stiffness{\ell}
		(\sigma_{\ell}^* )^{-1}
		\smoothingop{\ell}^{\nu_1}	,$$
		 Remark \ref{remark:telescoping} 
gives
$$
\|E\|_{L_2\to L_2}
\le
C( h_{\ell-1}^{-2}+\nu_1) h_{\ell-1}^{-d}\epsilon_{\ell-1} \|\multigridmatpert{\ell-1}^{\tau}  \| _{\ell_2\to\ell_2}.
$$
Using the Riesz property,  this  gives
\begin{eqnarray*}
\|\multigridoppert{\ell}  \| 
&\le&
\|\twogridoppert{\ell}   \| \\
&&+
C \left( h^{-2}+\nu_1) h^{-d}\epsilon \|\multigridoppert{\ell-1}^{\tau}  \| 
+
\left\|\sigma_{\ell}^*\prolmat{\ell} 		
		\multigridmatpert{\ell-1}^{\tau}    
		\stiffness{\ell-1}^{-1} 
		   \resmat{\ell}  
		      \stiffness{\ell}
		(\sigma_{\ell}^* )^{-1}
		\smoothingop{\ell}^{\nu_1}		
\right\|
\right).
\end{eqnarray*}
As in the proof of Theorem \ref{thm:main}, the last normed expression can be bounded as
$$\left\|\sigma_{\ell}^*\prolmat{\ell} 		
		\multigridmatpert{\ell-1}^{\tau}    
		\stiffness{\ell-1}^{-1} 
		   \resmat{\ell}  
		      \stiffness{\ell}
		(\sigma_{\ell}^* )^{-1}
		\smoothingop{\ell}^{\nu_1}		
\right\|\le \|\multigridoppert{\ell-1}^{\tau}  \| \|P_{\Xi_{\ell-1}}\smoothingop{\ell}^{\nu_1}\|
\le C \|\multigridoppert{\ell-1}^{\tau}  \|.$$
Because $( h_{\ell}^{-2}+\nu_1) h_{\ell}^{-d}\epsilon_{\ell}$ is bounded by assumption,
it follows that 
$$
\|\multigridoppert{\ell} \| 
	\le 
	\|\twogridoppert{\ell}   \| +
C_3  \|\multigridoppert{\ell-1}\|^{\tau} .
$$

As by assumption $ \epsilon_{\ell} \le h_{\ell}^{d+2}$ (no constants due to $h \le h_0$) holds for all $\ell\le L$, we 
obtain by Lemma \ref{lem:2step_pert}
\begin{align*}
\| \twogridmatpert{\ell}  \|_{2\to 2} &\le \| \twogridmat{\ell} -\twogridmatpert{\ell}  \|_{2\to 2} + \| \twogridmat{\ell}  \|_{2\to 2} \\
&\le 
C_{\text{Lem.\ref{lem:2step_pert}}}  (\nu_2 +\nu_1 + h_{\ell-1}^{-4}) h_{\ell-1}^{-(d+2)} \epsilon_{\ell-1} + \| \twogridmat{\ell}  \|_{2\to 2}\\
 &\le \min(C_1,\gamma C_2) + \| \twogridmat{\ell}  \|_{2\to 2}.
\end{align*}
We use Theorem \ref{thm:main} and choose a natural number $\nu^{\star}_1$ large enough such that  the inequality $C_{\text{Prop.\ref{prop:boundtwogridop}}} g(\nu^{\star}_1) \le C_4\min\left\{ \frac{\tau-1}{\tau} (\beta_{\text{Thm.\ref{thm:main}}} \tau)^{-\frac{1}{\tau-1}}, \frac{\tau-1}{\tau} \gamma \right\}$ is satisfied.
Thus, we obtain 
\begin{align*}
\| \twogridmatpert{\ell}  \|_{2\to 2} \le  \min(C_1,\gamma C_2) +C_4 \gamma.
\end{align*}
We have 
\begin{align*}
	C_1 + C_4 \gamma \le \frac{\tau-1}{\tau} (C_3 \tau)^{-\frac{1}{\tau-1}} \quad \text{and} \quad  C_2+C_4 \le \frac{\tau-1}{\tau}.
\end{align*}
Thus, we define
$$
\beta:=C_3 \quad \text{and} \quad 
{\alpha}
	:= \max_{\ell\le L} \|\twogridoppert{\ell}   \| \le \min\left\{\frac{\tau-1}{\tau} (\beta \tau)^{-\frac{1}{\tau-1}}, \frac{\tau-1}{\tau} \gamma   \right\}.
	 \qquad
 $$
Hence Lemma \ref{lem:recursivetau} applies and the result follows.
\end{proof}

\section{Truncated multigrid method}\label{sec:trunc}
In this section we consider truncating the stiffness, prolongation and restriction matrices in order to improve the computational
complexity of the method. Each such matrix has stationary, exponential off-diagonal decay, so by retaining the $(\xi,\eta)$ entry 
when $\dist(\xi,\eta)\le Kh_{\ell} |\log h_{\ell}|$, and setting the rest to zero,  guarantees a small perturbation error (on the order
of $\mathcal{O}(h_{\ell}^J)$, where $J\propto K$). 
This is made precise in Lemma \ref{lem:stiffnestrunc} and Lemma \ref{lem:resttrunc} below, with
the aid of the following lemma.
\begin{lemma}\label{lem:series_est}
Suppose $\Xi\subset \M$, $c>0$, and $r\ge 2q(\Xi)$. Then for any $\eta\in \M$, we have
$$\sum_{\substack{\xi\in\Xi \\ \dist(\xi,\eta)>r}} e^{-c \ \dist(\xi,\eta)} 
\le  \frac{\beta_{\M}}{\alpha_{\M}}
\left(\frac{r}{q}\right)^d e^{-c {r}} \left(\sum_{j=0}^{\infty}   (j+2)^de^{-cjq}\right).$$
\end{lemma}
\begin{proof}
The underlying set can be decomposed as $\{\xi \in \Xi\mid \dist(\xi,\eta) \ge r\} = \bigcup_{j=0}^{\infty} \mathcal{A}_j$, 
where $\mathcal{A}_j = \{\xi\in\Xi\mid r+  jq \le \dist(\xi,\eta) < r+(j+1)q\}$ has cardinality
$|\mathcal{A}_j| \le \frac{\alpha_{\M}}{\beta_{\M}} \bigl(\frac{r}{q} + (j+2)\bigr)^d$.
It follows that 
$$\sum_{\substack{\xi\in\Xi \\ \dist(\xi,\eta)>r}} e^{-c \ \dist(\xi,\eta)}
\le 
 \frac{\beta_{\M}}{\alpha_{\M}}
  e^{-cr} \sum_{j=0}^{\infty}  \left(\frac{r}{q} + (j+2)\right)^de^{-cjq}
 $$
 and the lemma follows from the fact that $\frac{r}{q} + j+2\le \frac{r}{q}  (j+2)$ for all $j\ge 0$.
\end{proof}
 \paragraph{Truncated stiffness matrix}
The exponential decay in Lemma \ref{lem:diag_decay} motivates the truncation of the stiffness matrix, see e.g. \cite[Eq. (8.1)]{NRW}
We define for positive $K$, the truncation parameter $r_{\Xi}:=K h |\log(h)|$
\begin{align}
\stiffnesstrunc{\Xi}{r_{\Xi}}\in \R^{|\Xi| \times |\Xi|  }\quad \text{with}\quad
	(\stiffnesstrunc{\Xi}{r_{\Xi}})_{\xi,\eta}:=
		\begin{cases}
			A_{\xi,\eta}=a( \chi_{\xi},\chi_{\eta}), & \dist(\xi,\eta)\le r_{\Xi},\\
			0, & \dist(\xi,\eta)> r_{\Xi}.
		\end{cases}
\end{align}
We note that $\stiffnesstrunc{\Xi}{r_{\Xi}}$ is symmetric if $\stiffness{\Xi}$ is symmetric. 
By construction and quasi-uniformity, we obtain 
$
	| \{\xi \in B(\eta, r_{\Xi}) \cap \Xi  \}|
	 	 \le \rho^{d}h^{-d}_{\Xi,\M}\frac{ \beta_{\M} }{\alpha_{\M}}(r_{\Xi}+q)^{d} \le
	 2 \frac{\beta_{\M}}{\alpha_{\M}}\rho^{d}|K\log h|^{d}.
$
By $h^{-d} \le C |\Xi|$, this yields

$|\{\xi \in \Xi \, \mid \, (\stiffnesstrunc{\Xi}{r_{\Xi}})_{\xi,\eta} \neq 0\}
 \le
 2\frac{ \beta_{\M}}{\alpha_{\M}}\rho^{d} d^{d} K^{d} (\log| \Xi|)^{d}$.

In particular, we obtain 
\begin{align}\label{eq:truncmatrixvectorprodcomplex}
	\operatorname{FLOPS}(\vec{x} \mapsto \stiffnesstrunc{\Xi}{r_{\Xi}}\vec{x})  \le C_{comp}  K^{d} \log(| \Xi|)^{d} |\Xi|
\end{align}
for the number of operations for a matrix vector multiplication with the truncated stiffness matrix, with 
$C_{comp} =  2\frac{ \beta_{\M}}{\alpha_{\M}}\rho^{d} d^{d} $.
\begin{lemma}\label{lem:stiffnestrunc}
With the global parameter $C_{\text{tr}}:=\frac{\beta_{\M}}{\alpha_{\M}}\sum_{n=1}^{\infty} \left( n+2\right)^d e^{-\frac{\nu }{\rho}n}$,
 the estimate 
\begin{align}\label{eq:boundAminustruncA}
\|\stiffness{\Xi}-\stiffnesstrunc{\Xi}{r_{\Xi}}\|_{2\to 2} \le C_{\text{tr}}   C_{\text{stiff}} (K\rho |\log(h)|)^d h^{\frac{\nu}{2}K -2}
\end{align}
holds. 
\end{lemma}
\begin{proof}
The proof for the first statement is essentially given in \cite[Prop 8.1]{NRW}.
Using Lemma \ref{lem:diag_decay},  we observe, by symmetry, that
	$\|\stiffness{\Xi}-\stiffnesstrunc{\Xi}{r_{\Xi}}\|_{p\to p}$ is controlled by the maximum of the $\ell_1$ and $\ell_{\infty}$ matrix norms,
	which can be controlled by row and column sums.  
	This leads to off-diagonal sums 
	$
	\max_{\eta \in \Xi} \sum_{\xi \in \Xi \cap B^{\complement}(\eta, r_{\Xi}) } \left|(\stiffness{\Xi})_{\xi,\eta} \right|$
	and $
	\max_{\xi \in \Xi} \sum_{\eta \in \Xi \cap B^{\complement}(\eta, r_{\Xi}) } \left|(\stiffness{\Xi})_{\xi,\eta} \right| .
	$
	Applying Lemma \ref{lem:series_est} with $r=r_{\Xi} = Kh|\log h|$ and $c=\frac{\nu}{2h}$ yields
\begin{align*}
        \|\stiffness{\Xi}-\stiffnesstrunc{\Xi}{r_{\Xi}}\|_{p\to p}
	&\le  
	C_{\text{stiff}} h^{-2}  \max_{\xi \in \Xi} \sum_{\eta \in \Xi \cap B^{\complement}(\xi, r_{\Xi}) }  e^{-\frac{\nu}{2} \frac{\dist(\xi,\eta)}{h}}\\
	&\le
	C_{\text{stiff}}\frac{\beta_{\M}}{\alpha_{\M}}\left(\frac{K h |\log h|}{q}\right)^d h^{\frac{K\nu}{2}-2}  \left(\sum_{j=0}^{\infty} (j+2)^d e^{-\frac{\nu}{2h} jq}\right).
\end{align*}
\end{proof}

 \paragraph{Truncated prolongation and restriction matrices}
We introduce  truncated prolongation matrices  
$\prolmattrunc{\ell}{r_{\ell}} \in \R^{n_{\ell} \times n_{\ell-1}}$,  with
\begin{align*} 
	(\prolmattrunc{\ell}{r_{\ell}})_{\xi \in \Xi_{\ell-1},\eta\in \Xi_{\ell}}:= \begin{cases}
			(\prolmat{\ell})_{\xi,\eta}= \chi^{(\ell-1)}_{\xi}(\eta), & \dist(\xi,\eta)\le r_{\ell},\\
			0, & \dist(\xi,\eta)> r_{\ell},
		\end{cases}
\end{align*}
where we use the notation $r_{\ell}:=r_{\Xi_{\ell-1}}$.
Likewise, we define 
$\resmattrunc{\ell}{r_{\ell}}:=\left(\prolmattrunc{\ell}{r_{\ell}}\right)^T$.
For the numerical costs, we obtain
\begin{align}\label{eq:costsprol}
\max\left\{ \operatorname{FLOPS}(\vec{x} \mapsto \prolmattrunc{\ell}{r_{\ell}} \vec{x}),\operatorname{FLOPS}(\vec{x} \mapsto \resmattrunc{\ell}{r_{\ell}} \vec{x}) \right\}  =\mathcal{O}\left( K^{d} \log(| \Xi_{\ell}|)^{d} |\Xi_{\ell}|\right),
\end{align}
where we use that  
$|\Xi_{\ell}|\sim h_{\ell}^{-d}=\gamma^{d} h^{d}_{\ell-1} \sim |\Xi_{\ell-1}|\sim \gamma^{d} |\Xi_{\ell-1}|$ due to (\ref{eq:gridconst}).

\begin{lemma}\label{lem:resttrunc}
We have 
\begin{align*}
	 \left\|\prolmat{\ell}-\prolmattrunc{\ell}{r_{\ell}} \right\|_{\ell_{2}(\Xi_{\ell-1}) \to \ell_{2}(\Xi_{\ell}) }
	  \le \Cpw C_{\text{trunc}} (K |\log(h_{\ell})|)^d h^{\frac{\nu}{2}K }_{\ell}.
\end{align*}
\end{lemma}
\begin{proof}
We proceed as in the proof of (\ref{eq:boundAminustruncA}).
We can estimate row and column sums of 
$\vec{E}:=\prolmat{\ell}-\prolmattrunc{\ell}{r_{\ell}}$ by (\ref{eq:ptwise_decay}), obtaining
$$
\left\|\vec{E} \right\|_{\infty\to \infty}
\le \Cpw \sum_{\xi\in\Xi_{\ell-1}\cap B^{\complement}(\eta,r_{\ell}) }  
e^{-\nu \frac{\dist(\eta,\xi)}{h_{\ell-1}}},$$
so, by Lemma \ref{lem:series_est}
with $r=r_{\ell} = Kh_{\ell}|\log h_{\ell}|$ and $c=\frac{\nu}{h_{\ell-1}}$, we have
$$
\left\|\vec{E} \right\|_{\infty\to \infty}
\le 
\Cpw
 \frac{\beta_{\M}}{\alpha_{\M}}
\left(  K\rho |\log h_{\ell-1}|\right)^d  (h_{\ell-1})^{K\nu} \left(\sum_{j=0}^{\infty}   (j+2)^de^{-\frac{\nu j}{ \rho}}\right).
$$
 Likewise,
$
\left\|\vec{E} \right\|_{1 \to 1}
\le  
\Cpw
\sum_{\eta\in\Xi_{\ell}\cap B^{\complement}(\xi,r_{\ell}) }  
e^{-\nu \frac{\dist(\eta,\xi)}{h_{\ell-1}}}$.
Lemma \ref{lem:series_est} yields this time 
with $r=r_{\ell} = Kh_{\ell}|\log h_{\ell}|$ and $c=\frac{\nu}{h_{\ell-1}}$, the estimate
$$
\left\|\vec{E} \right\|_{1\to 1}
\le 
\Cpw
 \frac{\beta_{\M}}{\alpha_{\M}}
\left(  K \frac{h_{\ell-1}}{q_{\ell}} |\log h_{\ell-1}|\right)^d  (h_{\ell-1})^{K\nu} \left(\sum_{j=0}^{\infty}   (j+2)^de^{-\frac{\nu j q_\ell }{ h_{\ell-1}}}\right).
$$
Thus, we get
\begin{align*}
\max \left\{\left\|\vec{E} \right\|_{\infty\to \infty},\left\|\vec{E}\right\|_{1\to 1} \right\}
&\le
\Cpw C_{\text{trunc}} (K \rho \gamma  |\log(h_{\ell-1})|)^d h_{\ell-1}^{\frac{\nu}{2}K }.
\end{align*}
Interpolation finishes the proof.
\end{proof}
 \paragraph{Truncated $\tau$-cycle}
We now consider the multigrid method using truncated versions of the stiffness, prolongation and restriction  matrices.
We denote this by $\mgmtrunc^{(\tau)}_{\ell}$, and use it  to solve 
(\ref{eq:trruncatedsystem}) with $\stiffnesspert{\ell} =\stiffnesstrunc{\ell}{r_{\ell}}$.
Lemmas \ref{lem:stiffnestrunc} and \ref{lem:resttrunc} show that conditions for Theorem \ref{thm:theorem_pert} are satisfied when $K$ is chosen sufficiently large.

\begin{theorem}\label{thm:truncated_complexity}
 If $\tau \gamma^d <1$, we obtain
\begin{align}\label{eq:mgmtaucost}
	\operatorname{FLOPS}(\vec{x} \mapsto \mgmtrunc^{(\tau)}_{\ell}(\vec{x},\vec{b}_{\ell} ))=\mathcal{O} (N_{\ell} \log(N_{\ell})^{d}).
\end{align}
\end{theorem}
\begin{proof}
Define the floating point operation count for the truncated 
multigrid\\ method by
${\tt M}_{\ell}:=\operatorname{FLOPS}(\vec{x} \mapsto \mgmtrunc^{(\tau)}_{\ell}(\vec{x},\vec{b}_{\ell} ))$.

By estimates (\ref{eq:truncmatrixvectorprodcomplex}) and (\ref{eq:costsprol}), 
the quantities
${\tt P}_{\ell} :=\operatorname{FLOPS}(\vec{x} \mapsto \prolmattrunc{\ell}{r_{\ell}} \vec{x})$,
${\tt R}_{\ell} :=\operatorname{FLOPS}(\vec{x} \mapsto \resmattrunc{\ell}{r_{\ell}} \vec{x}) $,
and
${\tt A}_{\ell} :=\operatorname{FLOPS}(\vec{x} \mapsto \stiffnesstrunc{\ell}{r_{\ell}} \vec{x}) $
are each bounded by $C K^{d} \log(| \Xi_{\ell}|)^{d} |\Xi_{\ell}|$.
Because each Jacobi iteration involves multiplication by a matrix with the same number of nonzero entries, we note
that 
$${\tt S}_{\ell}:=\operatorname{FLOPS}\bigl(x\mapsto J(x,b)\bigr) \le C K^{d} \log(| \Xi_{\ell}|)^{d} |\Xi_{\ell}|$$ as well.
From this, we have the recursive formula
$$
	{\tt M}_{\ell} =
	{\tt P}_{\ell} + 
	{\tt R}_{\ell} +
	{\tt A}_{\ell} +
	(\nu_1+\nu_2) {\tt S}_{\ell} 
	+\tau {\tt M}_{\ell-1}.
	$$

Applying (\ref{eq:truncmatrixvectorprodcomplex}) and (\ref{eq:costsprol}) gives
$
	{\tt M}_{\ell}\le 
	  C K^{d}   (\nu_1+\nu_2+3)\left( |\log h_{\ell} |^{d} h^{-d}_{\ell}\right)
	 +\tau {\tt M}_{\ell-1}.
$
	
By setting $w_{\ell}= 
\left( |\log h_{\ell} |/ h_{\ell}\right)^d$ and $\tilde{C}: = C  K^{d}   (\nu_1+\nu_2+3)$, we have the recurrence 
$$
{\tt M}_{\ell}\le 
	  \tilde{C}w_{\ell}
	 +\tau {\tt M}_{\ell-1} 
	 \quad
	 \longrightarrow \quad{\tt M}_{\ell} \le  \tilde{C}\sum_{k=0}^{\ell-1} w_{\ell-k}\tau^k +\tau^\ell {\tt M}_0
	 $$
Note that 
$
w_{\ell}\le (|\log h_0|+ \ell |\log \gamma|)^{d}    \gamma^{-d\ell} h_0^{-d} 
$, 
since  $h_{\ell} \le \gamma^{\ell} h_0$. 
By H{\"o}lder's inequality, we have  
$(|\log h_0|+ \ell |\log \gamma|)^{d}\le 2^{\frac{d-1}{d}}(|\log h_0|^d+ \ell |\log \gamma|^d)$,
which 
provides the estimate
$w_{\ell-k} \le  2^{\frac{d-1}{d}}h_0^{-d}   \gamma^{-d\ell}  \gamma^{dk}(|\log h_0|^d+\ell (1-k/\ell) |\log \gamma|^d)$.

Applying this to the above estimate for ${\tt M}_{\ell}$ gives, 
\begin{align*}
	{\tt M}_{\ell}
	&\le    \tilde{C}   h_0^{-d}\gamma^{-d\ell}   \left(|\log h_0|^d\sum_{k=0}^{\ell-1} (\tau \gamma^{d})^k
	+ |\log \gamma|^{d} \ell^d \sum_{k=0}^{\ell-1} (1-k/\ell)^d  (\tau \gamma^{d})^k  \right)\\
	&\phantom{\le}+\tau^\ell {\tt M}_0\\
	&\le
	   \tilde{C}   h_0^{-d}\gamma^{-d\ell}   \left|\log h_0 + \log \gamma^{\ell} \right|^d+\tau^\ell {\tt M}_0
	   \le \tilde{C} h^{-d} (|\log h|)^{d}+\tau^\ell {\tt M}_0.
	\end{align*}
The result follows by taking $(\gamma^\ell h_0)^{-d}\le C h^{-d}$ and $(|\log h_0|+\ell|\log \gamma| )  \sim |\log h|$.
\end{proof}

\begin{remark}\label{remark:final}
The kernel-based Galerkin problem $\stiffnesstrunc{L}{r_{L}}  \check{\vec{u}}^{\star}_{L}=\vec{b}_{L}$, can be solved stably to any precision $\epsilon_{\max}$,
 by iterating  $\mgmtrunc^{(\tau)}_{L}(\check{\vec{u}}_{L},\vec{b}_{L} )$, i.e., the truncated multigrid with $\tau\ge 2$ cycle. 
 Select $\gamma<1$ and fix $\nu_1$ so that Theorem \ref{thm:theorem_pert} holds.
 Let $\check{\vec{u}}_{L}^{(k+1)} =  \mgmtrunc^{(\tau)}_{L}(\check{\vec{u}}_{L}^{(k)},\vec{b}_{L} )$.
 If $k$ is the least integer satisfying  $\gamma^k\|\check{\vec{u}}_{L}^{\star}-\check{\vec{u}}_{L}^{(0)}\|_{\ell_2}<\epsilon_{\max}$, 
 then 
 \begin{equation*}
 k\sim \frac{1}{\log \gamma} \log\left(\frac{\epsilon_{\max}}{\|\check{\vec{u}}_{L}^{\star}-\check{\vec{u}}_{L}^{(0)}\|_{\ell_2}}\right).
 \end{equation*}
 Due to Theorem \ref{thm:truncated_complexity}, we obtain an overall complexity of 
 \begin{equation*}
 \mathcal{O} \left( \frac{1}{\log \gamma} \log\left(\frac{\epsilon_{\max}}{\|\check{\vec{u}}_{L}^{\star}-\check{\vec{u}}_{L}^{(0)}\|}\right) N_{L} \log(N_{L})^{d}\right).
 \end{equation*}
 We note that
 $
 \|\check{\vec{u}}_{L}^{\star}-\check{\vec{u}}_{L}^{(k)}\|_{\vec{A}_L}
 \sim 
 \|\sigma_L^*\left( \check{\vec{u}}_{L}^{\star}-\check{\vec{u}}_{L}^{(k)} \right)\|_{W_2^1}
 \le 
\CB h^{d/2-1} \|\check{\vec{u}}_{L}^{\star}-\check{\vec{u}}_{L}^{(k)}\|_{\ell_2},
$
 and since $d\ge2$, achieving $\|\check{\vec{u}}_{L}^{\star}-\check{\vec{u}}_{L}^{(k)}\|_{\vec{A}_L}<\epsilon_{\max}$ also requires only a fixed number of iterations. This is repeated below in  statement  (\ref{eq:mgm_iteration_count}).
 \end{remark}

Indeed, using $k$
steps  of the Conjugate Gradient method on the original system (\ref{eq:stiffness_equation}), would give error 
$\|\bar{\vec{u}}^{(k)}_{L}-\vec{u}^{\star}_{L}\|_{\vec{A}_{L}}\le \left(\frac{CN^{2/d}_{L}  -1}{CN^{2/d}_{L}  +1}\right)^k\|\vec{u}^{\star}_{L}-\vec{u}^{(0)}\|_{\vec{A}_{L}}$. 
Thus to ensure a tolerance of  $\|\bar{\vec{u}}^{(k)}_{L}-\vec{u}^{\star}_{L}\|_{\vec{A}_{L}}\le\epsilon_{\max}$, one would need
\begin{equation*}
k
\sim 
\frac{1}{\log (1-\tilde{C} N^{-2/d}_{L} )} \log\left(\frac{\epsilon_{\max}}{\|\bar{\vec{u}}^{(0)}_{L}-\vec{u}^{\star}_{L}\|}\right)
\sim 
\mathcal{O}\left( N^{2/d}_{L} \log\left(\frac{\epsilon_{\max}}{\|\bar{\vec{u}}^{(0)}_{L}-\vec{u}^{\star}_{L} \|}\right)\right) 
\end{equation*}
steps, where we use $\left(\frac{CN^{2/d}_{L}  -1}{CN^{2/d}_{L}  +1}\right) \sim \left( 1-\tilde{C} N^{-2/d}\right)$.

In contrast to this, the multigrid $W$-cycle requires only 
\begin{equation}\label{eq:mgm_iteration_count}
k\sim \left|\log\left(\frac{\epsilon_{\max}} {\|\check{\vec{u}}^{(0)}_{L}-\check{\vec{u}}^{\star}_{L}\|} \right)\right|
\end{equation}
iterations to achieve error
 $\|\check{\vec{u}}^{(k)}_{L}-\check{\vec{u}}^{\star}_{L}\|_{\vec{A}_{L}}\le \epsilon_{\max}$.
In fact, it reaches $\|\check{\vec{u}}^{(k)}_{L}-\check{\vec{u}}^{\star}_{L}\|_{\ell_2}\le \epsilon_{\max}$, which is a stronger constraint, within $k$ iterations. 
In particular, the number of iterations is independent of the size $N_L$ of the problem.

\section*{Acknowledegment}
The authors wish to thank the anonymous reviewers for their careful reading and helpful comments.

\bibliographystyle{elsarticle-num} 

\bibliography{mg}

\end{document}